%% file: main.tex
\documentclass[preprint,10pt]{elsarticle}

\usepackage[left=1cm,right=1cm,top=1.5cm,bottom=1cm,includeheadfoot]{geometry}
\usepackage{graphicx,xcolor}
\usepackage{amsthm,amssymb,amsmath,amsfonts}
\usepackage[bb=libus]{mathalpha}
\usepackage{comment, subfiles}
\usepackage{subcaption}
\usepackage{tikz}
\usetikzlibrary{arrows.meta,patterns}
\usepackage{booktabs}
\usepackage{multirow}
\usepackage{url} 

\input{defs}
\newtheorem{theorem}{Theorem}
\newtheorem{lemma}{Lemma}

\newcommand{\answer}[1]{#1}

\begin{document}

\begin{frontmatter}

\title{Splitting-strategies for arbitrary-order fully mixed finite element discretizations\\of the Biot equations}

\author[F]{Fleurianne Bertrand}
\author[B]{Jakub Wiktor Both}
\author[F]{Tugay Da\u{g}l{\i}}

\address[F]{Numerical Analysis of Partial Differential Equations, University of Technology Chemnitz, Chemnitz, Germany}
\address[B]{Department of Mathematics, University of Bergen, Bergen, Norway\footnote{Present address: Department of Physics and Technology, University of Bergen, Bergen, Norway}}

\begin{abstract}
We study the fully-mixed formulation of the Biot equations which is characterized by a symmetric coupling between flow and deformation \answer{while each subphysics has internally a saddle point structure}. This enables the use of stable mixed finite elements for each subproblem without a strong compatibility condition across the two subphysics. 
To exploit this flexibility while preserving the conservation structure of both subproblems, we consider fully mixed finite element methods in which the symmetry of the elastic stress tensor is enforced weakly \answer{often related to as five-field formulation}.
The resulting mixed formulation exhibits an overarching saddle-point structure whose stability is determined by suitable inf–sup conditions. Inf–sup stability is established for several families of discrete spaces of arbitrary order, leading to optimal a priori error estimates.
Iterative splitting strategies following the classical fixed-stress split with additional tuning are specifically investigated for the fully mixed formulation with proof of convergence and rates depending on the coupling strength. Contrary to previous analyses on coupled problems with a symmetric structure, we theoretically prove the efficacy of negative stabilization, consistent with Schur-complement ideas. Numerical results based on analytical solutions and the classical \answer{2D Mandel problem and a 3D footing problem} are presented supporting the theory.
\end{abstract}


\end{frontmatter}


\section{Introduction}
\label{sec:intro}

The two-way coupling between fluid flow and deformation of porous materials, commonly referred to as poromechanics, is a key component in a wide range of engineering applications, including subsurface energy and storage technologies, geomechanics, biomechanics, and material processing.
From a computational viewpoint, poromechanics models exhibit pronounced multi-physics and multi-parameter features, so that robustness with respect to physical parameters and discretization choices, and {reliability} via error control are central requirements for predictive simulations.
This perspective has motivated extensive, recent research on stable discretizations, parameter-robust solvers and preconditioners, and adaptive strategies driven by a posteriori error estimates; see, e.g., the overview and motivation in \cite{BertrandErnRadu:2021} as well as references therein. The most widely used mathematical model in this context is the linear, quasi-static Biot consolidation system \cite{Biot:1941,Biot:1955}, which couples mechanical equilibrium of the solid skeleton with mass conservation of the pore fluid.

\answer{The Biot equations admit a variety of formulations and discretizations. Coupling elasticity and porous media flow, both can be written in primal and dual form with suitable primary variables.}
A classical approach \answer{to formulate the Biot equations} is based on a two-field saddle-point formulation in displacement and pore pressure, for which stable conforming finite element choices such as the Taylor--Hood pair yield optimal a priori error estimates under standard assumptions, see \cite{MuradLoula:1992}. \answer{The saddle-point coupling, originating in Stokes-like coupling terms, can be also numerically stabilized through some form of pressure-stabilization~\cite{rodrigo2016stability}. However, if not intrinsically stable, a common drawback of discretizations of} this formulation is that stability and accuracy may deteriorate in regimes with incompressible constituents, or strongly varying parameters~\cite{haga2012causes}, in addition to lack of conservation qualities for mass and linear momentum, prompting the development of alternative formulations designed to be more robust.

A natural remedy is to employ mixed formulations \answer{for the two subphysics}, where additional variables for the Darcy velocity and/or stress are introduced and approximated directly. Moreover, they allow for schemes that honor the conservation properties. As a result, a range of equivalent formulations of the Biot equations exist and have been studied both in terms of discretizations and efficient numerical solution. 
With a strong focus on the approximation of the flow problem, the three-field formulation of poromechanics, built on a mixed formulation of the flow problem and adding the Darcy velocity as unknown, highlights local mass conservation and introduces a double-saddle point structure; see, e.g., \cite{PhillipsWheeler:2007}. \answer{Again, pressure stabilization is needed to ensure robustness in the presence of Stokes-type coupling~\cite{rodrigo2018new,camargo2021macroelement}.} On the contrary, local conservation of linear momentum can be achieved through a mixed \answer{(stress-based)} formulation of the elasticity problem, where the symmetry of the total Cauchy stress can be enforced strongly~\cite{ArnoldWinther:2002} or  weakly~\cite{FraeijsdeVeubeke1975,ArnoldFalkWinther:2007,guzman2010unified,quinelato2019full}, where for the former the construction of associated approximate function spaces considerably restricts the range of admissible discretizations and increases the computational effort. 

\answer{Fully-mixed formulations employ mixed formulations for both the flow and elasticity subproblems and provide a framework for preserving the fundamental conservation laws of poromechanics. Combining a mixed Darcy formulation of flow with a weakly symmetric stress-based formulation of elasticity yields five primary unknowns and thus the so-called five-field formulation of Biot's equations~\cite{Lee:2016,baerland2017weakly}. By introducing fluxes and stresses as primary variables, local mass conservation and linear momentum balance are enforced directly at the discrete level~\cite{Yi:2014}. At the same time, the flow and elasticity subproblems retain their respective saddle-point structures, allowing stable discretizations for the individual subproblems to be combined in a modular fashion to obtain stable schemes for the coupled Biot system with provable a priori error estimates~\cite{Lee:2016,ambartsumyan2020coupled,nordbotten2016stable}. This flexibility arises from the nature of the poromechanical coupling. In stress-based formulations, by using the constitutive law for effective stress, the coupling is expressed through volumetric stress and pore pressure in contrast to the classical displacement--pressure formulation. This reformulation, effectively results in a symmetric elasticity--flow coupling. This symmetry is closely related to the underlying gradient flow structure of poromechanics~\cite{both2019gradient}, which admits a block-separable minimization principle in terms of stress and pressure without additional constraints linking the two subphysics. The fully mixed formulation is furthermore particularly attractive for a posteriori error estimation. Since physically relevant stresses and fluxes are computed directly as primary variables and satisfy local conservation properties, equilibrated quantities are readily available. This enables the construction of Prager--Synge type estimates by combining these fields with reconstructed kinematically admissible displacement and pressure variables; see, e.g.,~\cite{ahmed2019adaptive,ahmed2020adaptive,RiedlbeckDiPietroErnGranetKazymyrenko:2017}.}

\answer{The iterative solution of poromechanics by decoupling flow and elasticity is a common solution strategy, both as a stand-alone fixed-point method~\cite{kim2011stability} and as the basis for parameter-robust preconditioning of Krylov subspace methods~\cite{white2016block}. The most widely used iterative coupling scheme is the fixed-stress split; we emphasize that alternative parameter-robust preconditioners have also been developed for specific discretizations, including the five-field formulation~\cite{baerland2017weakly}. The fixed-stress split is based on consecutively solving the flow problem at fixed volumetric stress and the elasticity problem at fixed pressure~\cite{settari1998coupled}. In its original setting of a two-field formulation (also true for three-field formulations), this physically motivated iteration requires a stabilization of the flow problem to ensure unconditional robustness when decoupling the underlying saddle-point structure~\cite{kim2011stability}. In contrast, for the five-field formulation, the iteration separates a block-symmetric coupling between stress and pressure variables. As a consequence, a naive decoupling of the two subphysics suffices and, in fact, coincides with the physical interpretation of the fixed-stress split~\cite{jayadharan2021domain,ahmed2020adaptive}, fixing the stress and pressure variables during respective flow and elasticity solves. In particular, the fixed-stress iteration can be identified with an alternating minimization scheme for the underlying variational principle~\cite{both2019gradient}, for which convergence follows from abstract minimization theory~\cite{both2022rate} in line with problem-specific analysis~\cite{jayadharan2021domain}. On the other hand, problem-specific analyses of the two- and three-field formulations have revealed potential to tune convergence properties through a careful choice and optimization of the explicitly employed stabilization~\cite{mikelic2013convergence,BothBorregalesNordbottenKumarRadu:2017,storvik2019optimization}. Yet, comparable analyses aimed at tuning convergence for the fully mixed formulation, with its symmetric elasticity--flow coupling, have not been presented.} We, however, highlight the study of other symmetric problems employing stabilization and in general advocating for (positive) stabilization~\cite{nuca2024splitting,brun2020iterative}, while intuitively a Schur-complement approach would instead suggest a (negative) destabilization. The latter idea has been successfully applied for an alternative fully dynamic poromechanics model with a symmetric coupling between displacement and flux variables~\cite{both2022iterative}, but not for the classical Biot equations.

In this work, we aim to highlight both the flexibility in designing spatial approximations as well as the potential of tuning convergence of the fixed-stress split by careful destabilization (negative stabilization), when considering the five-field formulation of the Biot equations. 
We emphasize, consistent with the symmetric coupling of the two subproblems as highlighted above, that inf-sup stability of the five-field formulations does not require a compatibility condition across the coupling of flow and deformation, resulting in an entire family of discretization schemes similar to the combination of classical Arnold-Falk-Winther and Raviart-Thomas elements, cf. ~\cite{ahmed2020adaptive,ambartsumyan2020coupled}. In addition, we define a stabilized splitting scheme, iteratively coupling flow and deformation. It is inspired by the fixed-stress split, but suggests negative stabilization $L^2$-type. The strength of destabilization is assessed through a theoretical convergence analysis, resulting in a simple expression which is consistent with previous analyses of the two-field formulation~\cite{mikelic2013convergence,BothBorregalesNordbottenKumarRadu:2017,storvik2019optimization}.

The remainder of the paper is organized as follows. In section \ref{sec:porousmedia} and section~\ref{sec:porousmediamixed}, the fully-mixed formulation of the Biot equations is presented in strong and weak formulation, resembling the five-field formulation with enforced weak symmetry of stress tensors. This structure allows for the use of standard $\Hdiv$-conforming finite element spaces for the stress and Darcy velocity variables and avoids the need for symmetric stress finite element spaces. Based on this framework, we list families of conforming mixed finite element discretizations of arbitrary order in section \ref{sec:fem}. In section \ref{sec:iterativecoupling}, we investigate the corresponding iterative solution strategies for the resulting coupled saddle-point systems and analyze the behavior of a tuned fixed-stress splitting scheme. \answer{Numerical experiments, including tests based on a manufactured solution, the classical 2D Mandel benchmark, and a 3D footing problem}, illustrate the theoretical findings in section \ref{sec:numerics}. The paper is closed with concluding remarks in section~\ref{sec:conclusion}.

\section{Stress-velocity-based porous media equations}
\label{sec:porousmedia}

We focus on the quasi-static Biot's model for three-dimensional consolidation, where the porous medium is characterized as linearly elastic, homogeneous, isotropic, and filled with a Newtonian fluid~\cite{Biot:1941,Biot:1955,coussy2004poromechanics}. 
The consolidation process is governed by a system of partial differential equations that integrate the dynamics of the fluid flow \answer{governed by the pore pressure $p$ and Darcy flux $\velocity$} with the elastic deformation $\bf u$ of the solid framework \answer{accompanied with Cauchy stress $\totalstress$}. \answer{We consider stress–velocity formulations which have previously been introduced in the literature~\cite{Yi:2014}.}

\subsection{\answer{Notation}}
\answer{Before presenting the model, we briefly introduce the notation used throughout the paper. Let $\Omega \subset \mathbb{R}^d$ ($d=2,3$) denote a connected domain with Lipschitz boundary. Let $\boldsymbol{n}$ denote the outward unit normal vector onto the boundary $\partial\Omega$. Scalar-valued functions are denoted by standard letters (e.g. $\testpressure$), vector-valued functions by bold letters (e.g. $\testvelocity$), and tensor-valued functions by bold underlined letters (e.g. $\teststress$) -- similarly for custom function spaces. Furthermore, we employ calligraphic letters to denote collections of variables and function spaces whenever a compact block notation is convenient. Motivated by the divergence-conforming formulations considered in this work, we introduce the corresponding function spaces. Let $L^2(\Omega)$ denote the space of square-integrable scalar-valued functions on $\Omega$ with associated norm $\|\cdot\|_0$ and inner-product $(\cdot,\cdot)$. Moreover, let $H^1(\Omega) := \left\{ v \in L^2(\Omega) \;:\; \nabla v \in L^2(\Omega)^d \right\}$, denote the space of functions with square-integrable weak gradient. We use $\div$ to denote the weak divergence operator, and define $\hdiv := \left\{ \mathbf{v}\in L^2(\Omega)^d \;:\; \div \mathbf{v}\in L^2(\Omega) \right\}$, the space of vector-valued functions with square-integrable weak divergence. Moreover, $\hdiv^d := \left\{ \teststress\in L^2(\Omega)^{d\times d} \;:\; \text{each row of }\teststress\text{ belongs to }\hdiv \right\}$, denotes the corresponding space of tensor-valued functions $\teststress$, endorsed with the $\hdiv^d$ norm (similarly for vector valued-functions)
\[
  \|\teststress\|_{H(\div)}^2 := \|\teststress\|_0^2 + \|\div\teststress\|_0^2.
\]
Finally, $\bI$ denotes the identity tensor, and for a general tensor field
$\teststress=(\tau_{ij})$, its trace is $\tr\teststress := \sum_{i=1}^{d}\tau_{ii}$.
}

\subsection{Continuous first-order stress-velocity Biot system}

The poroelastic problem is based on the reference configuration of the undeformed state described by the domain $\Omega$. Following the assumptions of the theory of linear elasticity, the effective stress $\stress$ is related to the strain tensor 
$\strain:=\frac{1}{2}\left(\nabla \displacement+(\nabla \displacement)^\top\right)$
by the constitutive equation 
$\stress=2 \mu \strain+\lambda \operatorname{div}(\displacement) \bI$
where $\lambda$ and $\mu$ denote the Lam\'e coefficients.
Considering an external force $\bf g$,
the equilibrium equation for the porous medium reads
\begin{align}
\label{eq:equilibrium}
-\operatorname{div} \stress+\alpha \nabla p= \boldsymbol{g},
\end{align}
where $p$ denotes the pore pressure and $\alpha$ is the Biot–Willis constant. In the context of mixed finite element methods, it is crucial to consider the equilibrium equation in divergence form. To this end, we introduce the Cauchy stress tensor 
$\totalstress := \stress - \alpha p \bI$. Since 
$\nabla p = 
\div (p \bI)$, \answer{this will later allow us to seek for $\totalstress \in \hdiv^d$, as \eqref{eq:equilibrium}
can now be written as}
\begin{align}
\label{eq:total}
-\div \totalstress= \boldsymbol{g}.
\end{align}
\answer{With $\mathbb{A}$ denoting the inverse of the fourth-order stiffness tensor associated to Saint Venant-Kirchhoff materials}
\begin{align}
\label{eq:inverse-stiffness}
    \answer{\mathbb{A} \totalstress := 
    \frac{1}{2 \mu}\left(\totalstress-
    \frac{\lambda}{2\mu+d\lambda}(\operatorname{tr} \totalstress) {\bI} \right)},
\end{align}
\answer{the constitutive relation for the total stress can be inverted with respect to the strain}

\begin{align}
    \label{eq:constitutive}
    \answer{\strain 
    = \mathbb{A} \stress 
    = \mathbb{A} \totalstress + \mathbb{A} ( \alpha p \bI)
    = \mathbb{A} \totalstress + \frac{\alpha}{2\mu+d\lambda} p {\bI}.}
\end{align}
On the mass conservation side, under a forced fluid extraction or injection process $f$, the variation of \answer{the displacement-based} fluid content and the percolation velocity of the fluid $\velocity$ are related in divergence form by
\begin{align}
\label{eq:mass}
\frac{\partial}{\partial t}\left(\storage p+\alpha \div \displacement\right)+\div \velocity=f,
\end{align}
where $\storage$ is called storage coefficient.
In order to obtain a symmetric variational formulation, it is of interest to reformulate \eqref{eq:mass} in terms of $\totalstress$.
To this end, \answer{using $\div \displacement = \operatorname{tr}\strain$ together with \eqref{eq:inverse-stiffness}--\eqref{eq:constitutive}, equation~\eqref{eq:mass} with a stress-based fluid content reads}
\begin{align}
\label{eq:mass-stress}
\frac{\partial}{\partial t}\left(\left(\storage+
\frac{d\alpha^2}{2\mu+d\lambda} 
\right)p+
\frac \alpha {2\mu+d\lambda} \operatorname{tr} \totalstress\right)+\div \velocity=f.
\end{align}
Using the \answer{scalar-valued permeability $\kappa$, assuming for simplicity isotropic materials and 
$0 <\underline{\kappa} \leq \kappa(x)\leq   \bar{\kappa}$}, Darcy's law linearly relates the volume flow rate of the fluid to the pressure gradient as follows 
\begin{align}
    \label{eq:Darcy}
\answer{\velocity}=-\kappa \nabla p.
\end{align}

\answer{In order to keep the notation simple, we restrict ourselves to homogeneous boundary conditions, while the subsequently presented analysis extends in a standard way to non-homogeneous boundary data and the numerical examples presented later employ non-trivial boundary conditions (see sections \ref{sec:mandel} and \ref{subsec:footing_3d}).
We consider a decomposition of the boundary $
\partial\Omega=\Gamma_\mathrm{f}\cup\Gamma_\mathrm{p}
=\Gamma_\mathrm{t}\cup\Gamma_\mathrm{d}$ with the index addressing the constrained field (flux, pressure, traction, and displacement). Then we consider the boundary conditions}
\begin{align}
\label{eq:bc}
\answer{\totalstress {\bf n}} \;=\; \boldsymbol 0 \quad \text{on }\Gamma_\mathrm{t}
\qquad
\displacement \;=\; \boldsymbol 0 \quad \text{on }\Gamma_\mathrm{d},
\qquad
p \;=\; 0 \quad \text{on }\Gamma_\mathrm{p},
\qquad
\velocity\cdot  {\bf n} \;=\; 0 \quad \text{on }\Gamma_\mathrm{f},
\end{align}
\answer{allowing for Korn and Poincar\'e constants $C_\mathrm{K}$ and $C_\mathrm{P}$. Finally, we close the model problem with conforming initial data $\totalstress_0, p_0 \in L^2(\Omega)$  for the total stress and fluid pressure, defining the fluid content,}
\begin{align*}
\answer{\totalstress \;=\; \totalstress_0 \quad \text{on }\Omega,}
\qquad
\answer{p \;=\; p_0 \quad \text{on }\Omega.}
\end{align*}
\subsection{\answer{Semi-discrete first-order stress-velocity Biot system}}
In order to focus on the derivation of the mixed formulation of the first‑order stress–velocity Biot system and the spatial discretization in the next sections \ref{sec:porousmediamixed} and \ref{sec:fem}, we employ the commonly used implicit Euler scheme, although any higher‑order implicit method could be used without affecting the developments that follow. To this end, let $0 = t_0 \leq t_1 \leq ... \leq t_N = T$ denote an equidistant partition of the time interval with time step size $\Delta t = t_n - t_{n-1}$, $n=1,...,N\in \mathbb{N}$. \answer{With a focus onto solving a single discrete time step, we introduce the rescaled Darcy velocity $\integratedvelocity := \Delta t\,\velocity$}. \answer{Moreover, time-discrete quantities are denoted by superscripts when convenient, e.g., $(\totalstress^{n-1}, p^{n-1})$ for data at the previous time level. For brevity, superscripts, associated with the current time level, are omitted throughout the manuscript.} 
\answer{Focusing on the semi-discrete strong form for now, we consider the canonical function spaces with suggested regularity motivated by the governing equations.} \answer{Given $(\totalstress^{n-1},p^{n-1})\in \hdiv^d \times H^1(\Omega)$ from the previous time level, starting from initial conditions $\totalstress^{0}=\totalstress_0$ and $p^0=p_0$, the time-discrete Biot system at time step $n$ consists of finding $(\totalstress,p,\displacement,\integratedvelocity)\in 
\hdiv^d \times H^1(\Omega)  \times H^1(\Omega)^d \times \hdiv$
such that}
\begin{subequations}\label{eq:strong-form-sigma} \begin{alignat}{2} -\div \totalstress &= \mathbf g &&\quad \text{in }\Omega, \label{eq:strong-form-sigma-momentum} \\[0.9em] \mathbb{A} \totalstress + \tilde \alpha p \,\bI - \strain &= \boldsymbol{0} &&\quad \text{in }\Omega, \label{eq:strong-form-sigma-venant} \\[0.9em] \tilde \alpha\,\tr(\totalstress) + \tilde c_0\, p + \div \integratedvelocity &= \tilde{f} &&\quad \text{in }\Omega, \label{eq:strong-form-sigma-mass} \\[0.9em] \nabla p + \Delta t^{-1} \kappa^{-1}\,\integratedvelocity &= 0 &&\quad \text{in }\Omega, \label{eq:strong-form-sigma-darcy} \end{alignat} \end{subequations}
\answer{accompanied by boundary conditions of the same form as~\eqref{eq:bc}}
\begin{align}
\label{eq:bc-discrete}
\answer{\totalstress {\bf n}} \;=\; \boldsymbol 0 \quad \text{on }\Gamma_\mathrm{t},
\qquad
\displacement \;=\; \boldsymbol 0 \quad \text{on }\Gamma_\mathrm{d},
\qquad
p \;=\; 0 \quad \text{on }\Gamma_\mathrm{p},
\qquad
\integratedvelocity\cdot  {\bf n} \;=\; 0 \quad \text{on }\Gamma_\mathrm{f}.
\end{align}
\answer{To both simplify the notation as well as better highlight the symmetric poromechanical coupling, we have used the abbreviations}
\[
\tilde \alpha := \frac{\alpha}{2\mu+d\lambda},
\qquad 
\tilde c_0 := \storage + \frac{d\alpha^2}{2\mu+d\lambda},
\qquad
\answer{\integratedvelocity := \Delta t\,\velocity,}
\qquad
\tilde{f} := \Delta t f  
+ \tilde c_0\, p^{n-1}
+ \tilde \alpha\,\tr(\totalstress^{n-1}).
\]

\section{\answer{A fully-mixed formulation with weak symmetry}}
\label{sec:porousmediamixed}

\answer{Based on the strong form presented in the previous section, we focus on the weak formulation of the fully-mixed formulation of the Biot equations. For this, we choose to enforce symmetry of the total stress in a weak sense, as also presented in~\cite{Lee:2016,baerland2017weakly}.}

\subsection{\answer{Weak formulation with weak symmetry}}

It is well known that a variational formulation based on the Hellinger-Reissner principle can be derived from the constitutive equation \eqref{eq:constitutive}
via the Lagrange multiplier method by eliminating the constraints of equilibrium equations. However, the presence of symmetric gradient $\strain$ requires a symmetric test space for the integration by parts when testing this equation with a test stress $\teststress \in \hdiv^d$. To avoid later on the construction of a symmetric space, we consider the additional variable $\boldsymbol{\xi} = \frac{1}{2} \nabla \times \displacement$, which is a scalar for $d=2$ and a vector for $d=3$. 
It holds
\[
\strain = \nabla \mathbf u - \boldsymbol{\xi}\cdot\chi ,
\]
where the matrix $\boldsymbol{\xi}\cdot\chi \in \mathbb{R}^{d\times d}$ is defined componentwise by
\[
(\boldsymbol{\xi}\cdot\chi)_{ij}
=
\varepsilon_{ij}\,\xi
\text{ if } d=2,
\text{ and }
\displaystyle \sum_{k=1}^3 \varepsilon_{ijk}\,\xi_k
\text{ if } d=3,
\qquad i,j=1,\dots,d
\]
with the Levi--Civita symbol
$\varepsilon$.
Starting from~\eqref{eq:strong-form-sigma-venant}, the reformulated constitutive equation
$    \mathbb{A} \totalstress = 
     \nabla \mathbf{u}-\boldsymbol{\xi} \cdot \chi
 - \tilde \alpha p \bI
$
can now be tested with a non-symmetric test stress $\teststress \in \hdiv^d$ to obtain
\begin{align}
    \label{eq:constitutiveweak}
    (\mathbb{A} \totalstress,\teststress)
    + {\tilde \alpha } ( p,\tr\teststress)
     +( \displacement,\div \teststress)
     + (\boldsymbol{\xi} \cdot \chi,\teststress)
 =0,
\end{align}
where we have integrated by parts the second term. 
In addition, testing the mass conservation equation in~\eqref{eq:strong-form-sigma-mass} with $q \in L^2(\Omega)$ leads to
\begin{align}
\label{eq:massweak}
\tilde c_0 (p,q)
+\tilde \alpha
\left(
\operatorname{tr} \totalstress,q
\right)
+ (\div \integratedvelocity,q) = \langle \answer{\tilde f},q\rangle 
\end{align}
for all $q \in {{ L}}^2(\Omega)$. \answer{We highlight the symmetric structure of the stress-pressure coupling terms. The resulting block-symmetric bilinear forms correspond to a minimization problem for the energy $(\totalstress,p) \mapsto \frac{1}{2} (\mathbb{A}\totalstress,\totalstress) + \tilde \alpha (p, \operatorname{tr} \totalstress) + \frac{\tilde c_0}{2}(p,p)$. The additional contributions enter through external data and dissipation associated to fluid fluxes. Lagrange multipliers enter as the minimization is subject to conservation of linear momentum, symmetry of the total stress tensor and mass conservation; we refer to~\cite{both2019gradient} for a detailed discussion of the underlying gradient flow perspective on poromechanics, the corresponding minimization principles, and the roles of primal and dual variables. Indeed, the displacement $\displacement \in L^2(\Omega)^d$ enters as Lagrange multiplier and} corresponds to the weak form of \eqref{eq:equilibrium} as it holds
\begin{align}
    \label{eq:Equilibriumweak}
-(\div \totalstress,\testdisplacement)=({\bf g},\testdisplacement)
\end{align}
for all $\testdisplacement \in L^2(\Omega)^d$.
The Lagrange parameter $\vorticity \in L^{d(d-1)/2}(\Omega)$ corresponds to the symmetry constraint
$\operatorname{as}\totalstress = 0$, which in weak form reads 
\begin{align}
    \label{eq:symmetryweak}
(\totalstress,\testvorticity \cdot \chi ) = 0\end{align} for all $\testvorticity \in L^{d(d-1)/2}(\Omega)$.
The weak imposition of equilibrium and symmetry is consistent,
since divergence and skew–symmetric components of stresses in $\hdiv^d$
can be prescribed independently. More precisely, for any
$(\testdisplacement,\testvorticity)
\in L^2(\Omega)^d \times L^2(\Omega)^{d(d-1)/2}$
there exists $\teststress\in \hdiv^d$ such that
\begin{align}
\label{eq:infsupelasticity}
\div\teststress = \testdisplacement,
\qquad
(\teststress,\testvorticity\cdot\chi)
\gtrsim
\|\testvorticity\|_0^2,
\text{ and }
\|\teststress\|_{\answer{H(\div)}}
\lesssim
\|\testdisplacement\|_0
+
\|\testvorticity\|_0.
\end{align}
\answer{In the above mentioned minimization perspective, the scaled Darcy velocity $\integratedvelocity \in \hdiv$ enters as Lagrange multiplier originating from a mass conservation constraint, yet it also is associated with quadratic dissipation in form of the weak formulation of Darcy's law. Testing~\eqref{eq:strong-form-sigma-darcy} with a test function $\testvelocity \in \hdiv$ yields}
\begin{align}
    \label{eq:Darcyweak}
\answer{\frac{1}{\Delta t}(\kappa^{-1}\integratedvelocity,\testvelocity) - ( p,\div\testvelocity) =0.}
\end{align}

\answer{The weak formulation of~\eqref{eq:strong-form-sigma} with integrated boundary conditions~\eqref{eq:bc-discrete} lead to the following canonical functions spaces for the different primary variables}
\begin{align*}
\bf{X} &:= \Stress \times \Pressure \times \Displacement \times \Vorticity \times \Velocity ,\\
\Stress &:= \{\teststress\in \answer{\hdiv^d}:\ \teststress {\bf n}=\boldsymbol 0\text{ on }\Gamma_\mathrm{t}\},\\
\Pressure &:= L^2(\Omega),\\
\Displacement &:= L^2(\Omega)^d,\\
\Vorticity &:= L^2(\Omega)^{d(d-1)/2},\\
\Velocity &:= \{\testvelocity\in \answer{\hdiv}:\ \testvelocity\cdot {\bf n}=0\text{ on }\Gamma_\mathrm{f}\}.
\end{align*}
We can now combine the equations \eqref{eq:constitutiveweak},\eqref{eq:Darcyweak},\eqref{eq:Equilibriumweak},\eqref{eq:symmetryweak} and \eqref{eq:massweak} to obtain the final \answer{semi-discrete} variational formulation: seek
$(\totalstress,p,
\displacement,
\vorticity,
\integratedvelocity
 ) \in {\bf X}$ such that
\begin{equation}
\label{eq:weakformulation}
\begin{aligned}
(\mathbb{A} \totalstress,\teststress)
 &+ {\tilde \alpha } ( p,\tr\teststress)
  &+( \mathbf{u},\div \teststress)
&+ (\boldsymbol{\xi} \cdot \chi,\teststress)
 &&=0
 \\
 \tilde \alpha
\left(
\operatorname{tr} \totalstress,q
\right)&+
\tilde c_0\left(p,q\right) &&&
+(\div \integratedvelocity,q)&=\langle \tilde{f},q\rangle 
\\
(\div \totalstress,\testdisplacement)&&&&&=-({\bf g},\testdisplacement)
\\
(\totalstress,\testvorticity\cdot \chi ) 
&&&&& = 0 
\\
 &
\ \ \ ( p,\div\testvelocity)
&&&-\frac{1}{\Delta t}(\kappa^{-1}\integratedvelocity,\testvelocity)
&=0  
\end{aligned}
\end{equation}
for all 
$(\teststress,q,
\testdisplacement,
\testvorticity,
\testvelocity
 ) \in {\bf X}$. 

\subsection{\answer{General saddle-point view}}

\answer{To facilitate a unified analysis of stability, we recast the weak formulation~\eqref{eq:weakformulation} as an abstract saddle-point problem, allowing the use of general inf--sup theory.} In fact, the system~\eqref{eq:weakformulation} can be cast into the following saddle-point problem: seek \answer{$(\tupleu,\tuplep) \in \tupleV \times \tupleQ$} such that
\begin{equation}
\label{eq:genericsaddlepoint}
\begin{aligned}
\mathcal A(\tupleu, \tuplev)+
\mathcal B(\tuplev, \tuplep) & =
\mathcal F(\tuplev)
\\
\mathcal B(\tupleu, \tupleq)  - 
\mathcal C(\tuplep, \tupleq)& =
\mathcal G( \tupleq)
\end{aligned}
\end{equation}
holds for all $ (\tuplev,\tupleq) \in 
\tupleV \times \tupleQ$.
\answer{For this, we separate the generalized stresses and Lagrange-type parameters in multi-vector variables}
\begin{alignat*}{2}
\tupleV &:= \Stress \times \Pressure, &\qquad 
\tupleQ &:= \Displacement \times \Vorticity \times \Velocity,\\
\tupleu &:=(\totalstress, p), &\qquad 
\tuplep &:= (\displacement, \vorticity, \integratedvelocity), \\
\tuplev &:=(\teststress,q), &\qquad 
\tupleq &:= (\testdisplacement,\testvorticity,\testvelocity)
\end{alignat*}
\answer{with the canonical norms}
\begin{align*}
\answer{\|\tuplev \|_{\tupleV}}^2:= 
\answer{\| (\teststress,q) \|^2_{\tupleV}
=
\| \teststress \|^2_{{H(\div)}}
+
\| q \|^2_{0}
\text{ and }
}
\answer{
\|\tupleq \|^2_{\tupleQ}
}:= 
\answer{
\|(\testdisplacement,\testvorticity,\testvelocity) \|^2_{\tupleQ}
=
\|\testdisplacement \|^2_{0}
+
\|\testvorticity \|^2_{0}
+
\|\testvelocity \|^2_{{H(\div)}}
}.
\end{align*}

\answer{Consequently, the bilinear and linear forms collect the respective terms in the following way}
\begin{equation}
\begin{alignedat}{1}
\mathcal A(\tupleu,\tuplev)
&:= (\mathbb{A}\totalstress,\teststress)
   + \tilde c_0 (p, q)
   + \tilde\alpha (p,\tr\teststress)
   + \tilde\alpha (q,\tr\totalstress),\\
\mathcal B(\tuplev,\tuplep)
&:= (q,\div\integratedvelocity)
   + (\displacement,\div\teststress)
   + (\boldsymbol{\xi}\cdot\chi,\teststress),\\
\mathcal C (\tuplep,\tupleq)
&:= \frac{1}{\Delta t}(\kappa^{-1}\integratedvelocity,\testvelocity),\\
\mathcal F(\tuplev) &:= \langle \tilde{f},q\rangle,
\\
\mathcal G(\tupleq) &:= -({\bf g},\testdisplacement).
\end{alignedat}
\end{equation}
\answer{The bilinear forms $\mathcal A$ and $\mathcal C$ are continuous on $\tupleV\times\tupleV$ and $\tupleQ\times\tupleQ$, respectively, with continuity constants $(2\mu)^{-1}+\widetilde c_0+2\sqrt d\,\widetilde\alpha$ and $(\Delta t\,\underline\kappa)^{-1}$.}
The bilinear form $\mathcal B$ encodes the constraint structure of the mixed formulation. Its associated kernel spaces
\begin{align}
H &:= 
\{\tuplep\in\tupleQ :
\mathcal B(\tuplev,\tuplep)=0
\ \forall \tuplev\in\tupleV\}
=
\{(\bf{0}, \bf{0}, \integratedvelocity)\in\tupleQ :
\div\integratedvelocity=0\}\\
\text{ and }
K &:= 
\{\tuplev\in\tupleV :
\mathcal B(\tuplev,\tuplep)=0
\ \forall \tuplep\in\tupleQ\}
=\{(\teststress,0)\in\tupleV :
\div\teststress=0,\ 
\operatorname{as}(\teststress)=0\}
\end{align}
identify precisely those variables that remain uncoupled by $\mathcal B$ and therefore govern the stability of the saddle–point problem \eqref{eq:genericsaddlepoint}. On the kernel spaces, the energy forms decouple: on $K$, the form $\mathcal A$ acts only on the symmetric, divergence-free stress and is coercive with constant \answer{$(2\mu\,C_\mathrm{K}^2)^{-1}$}, whereas on $H$, the Darcy dissipation $\mathcal C$ acts only on the divergence-free velocity and is coercive with constant \answer{$(\Delta t \overline\kappa)^{-1}$}.  

It remains to establish the inf–sup stability of $\mathcal B$
on the orthogonal complement
\[
H^\perp
=
\Big\{
\tupleq = (\displacement,\vorticity,\integratedvelocity)\in\tupleQ :
(\integratedvelocity,\mathbf z)=0
\ \forall \mathbf z\in\hdiv \text{ s.t. } \div\mathbf z=0
\Big\}.
\]
Since the divergence operator on $\hdiv$ has closed range on bounded Lipschitz domains, its restriction to the $L^2$–orthogonal complement of the divergence–free subspace is injective with bounded inverse. Consequently, on $H^\perp$ the velocity component satisfies
\[
\|\integratedvelocity\|_{H(\div)}
\lesssim
\|\div\integratedvelocity\|_0.
\]
Together with the trivial inequality
\(
\|\div\integratedvelocity\|_0 \le \|\integratedvelocity\|_{H(\div)},
\)
this yields the norm equivalence
\[
\|\integratedvelocity\|_{H(\div)}
\simeq
\|\div\integratedvelocity\|_0
\quad\text{on } H^\perp.
\]
Let $\tupleq=(\displacement,\vorticity,\integratedvelocity)\in H^\perp$.
Choose $q_*=\div\integratedvelocity$
and select $\teststress_*\in\hdiv^d$ according to
\eqref{eq:infsupelasticity}
such that
\(
\div\teststress_*=\displacement.
\)
Then
\[
\mathcal B((\teststress_*,q_*),\tupleq)
=
(q_*,\div\integratedvelocity)
+
(\displacement,\div\teststress_*)
+
(\vorticity\cdot\chi,\teststress_*)
\gtrsim
\|\div\integratedvelocity\|_0^2
+
\|\displacement\|_0^2
+
\|\vorticity\|_0^2 .
\]
Moreover,
\[
\|(\teststress_*,q_*)\|_{\tupleV}
\lesssim
\|\displacement\|_0
+
\|\vorticity\|_0
+
\|\div\integratedvelocity\|_0 .
\]
Using the norm equivalence on $H^\perp$, we conclude that
\[
\inf_{0\neq \tupleq\in H^\perp}
\sup_{0\neq \tuplev\in\tupleV}
\frac{\mathcal B(\tuplev,\tupleq)}
{\|\tuplev\|_{\tupleV}\,\|\tupleq\|_{\tupleQ}}
\ge \beta
\]
for some constant $\beta>0$ independent of the data.
We are therefore in the setting of a perturbed saddle–point problem
with coercivity of $\mathcal A$ on $K$,
coercivity of $\mathcal C$ on $H$,
and inf–sup stability of $\mathcal B$ on $H^\perp$.
This yields the following well–posedness result.
\begin{theorem}
\label{thm:wellposedness}
Let $\Omega$ be a bounded, connected Lipschitz domain.
Assume that the Lamé parameters $\mu,\lambda$,
the storage coefficient $\tilde c_0$,
and the Biot parameter $\tilde\alpha$ are strictly positive,
and that the permeability tensor $\boldsymbol\kappa$
is symmetric and uniformly elliptic.
Given $\tilde{f}\in L^2(\Omega)$ and ${\bf g}\in L^2(\Omega)^d$,
 the saddle–point problem
\eqref{eq:genericsaddlepoint}
admits a unique solution
$(\tupleu,\tuplep)\in\tupleV\times\tupleQ$.
Moreover, there exists a constant $C>0$ such that the stability estimate
\[
\|\tupleu\|_{\tupleV}
+
\|\tuplep\|_{\tupleQ}
\le
C\Big(
\|\mathcal F\|_{\tupleV'}
+
\|\mathcal G\|_{\tupleQ'}
\Big)
\]
holds.
\end{theorem}
\begin{proof}{\answer{
This is Theorem 4.3.1. in \cite{BoffiBrezziFortin:2013}. With the constants computed above we obtain $C=\min(C_f,C_g)$ with}}
\answer{
\begin{align*}
C_f &\le \frac{a_{f,1}}{\Delta t} + \frac{a_{f,1/2}}{\sqrt{\Delta t}} + a_{f,0},
\qquad
C_g \le \frac{a_{g,1/2}}{\sqrt{\Delta t}} + a_{g,0}
       + \frac{\sqrt{2\,\overline\kappa\,\|\mathcal A\|}}{\beta}\sqrt{\Delta t}
       + 2\,\overline\kappa\,\Delta t,
\end{align*}
where all coefficients are independent of $\Delta t$ and $\lambda$:
\begin{align*}
a_{f,1}   &= \frac{1+2C_K\sqrt{2\mu\,\|\mathcal A\|}+4\mu C_K^2\,\|\mathcal A\|}{\underline\kappa\,\beta^2},
&
a_{f,1/2} &= \frac{2C_K\sqrt{\mu}}{\sqrt{\underline\kappa}\,\beta}
           + \frac{\big(\sqrt{\underline\kappa}+\sqrt{\overline\kappa}\big)\sqrt{\|\mathcal A\|}}
                  {\underline\kappa\,\beta^2}
             \Big(1+3C_K\sqrt{2\mu\,\|\mathcal A\|}\Big),
\\
a_{f,0}   &= 4\mu C_K^2
           + \frac{\sqrt{\underline\kappa}+\sqrt{\overline\kappa}}{\sqrt{\underline\kappa}\,\beta}
             \Big(1+3C_K\sqrt{2\mu\,\|\mathcal A\|}\Big),
&
a_{g,1/2} &= \frac{\big(\sqrt{\underline\kappa}+3\sqrt{\overline\kappa}\big)\sqrt{\|\mathcal A\|}}
                  {\underline\kappa\,\beta^2}
             \Big(1+C_K\sqrt{2\mu\,\|\mathcal A\|}\Big),
\\
a_{g,0}   &= \frac{\sqrt{\underline\kappa}+3\sqrt{\overline\kappa}}{\sqrt{\underline\kappa}\,\beta}
             \Big(1+C_K\sqrt{2\mu\,\|\mathcal A\|}\Big)
           + \frac{\|\mathcal A\|\big(\underline\kappa
             +2\sqrt{\underline\kappa\,\overline\kappa}
             +2\,\overline\kappa\big)}{\underline\kappa\,\beta^2}.
\end{align*}
Here we abbreviated
$\|\mathcal A\| = (2\mu)^{-1}+\widetilde c_0+2\sqrt d\,\widetilde\alpha$ for brevity.
}
\end{proof}

\section{Finite element formulation}
\label{sec:fem}
We now consider a shape-regular finite element triangulation 
$\tri$ of $\Omega$. 
The mesh induces the set of element faces (or edges in $d=2$)
$\mathcal E$, which is assumed to respect the boundary decomposition.
For each element $T\in\tri$ we denote by $h_T:=\operatorname{diam}(T)$
its diameter, and define the global mesh size by
\[
h := \max_{T\in\tri} h_T .
\]
A conforming Galerkin discretization of order $k\geq 0$ of the mixed formulation
\eqref{eq:genericsaddlepoint} consists in selecting finite element
subspaces
\[
\answer{\bf{X}_\tri := \tupleVT \times \tupleQT\subset \bf{X},\qquad 
\tupleVT := \StressT \times \PressureT \subset \tupleV = \Stress \times \Pressure,
\qquad
\tupleQT := \DisplacementT \times \VorticityT \times \VelocityT \subset \tupleQ = \Displacement \times \Vorticity \times \Velocity,}
\]
\answer{exemplified in the next subsection.} Altogether, by employing the Galerkin method to~\eqref{eq:weakformulation}, the fully-discrete solution for a single time step reads \answer{(again dropping superscripts for the current time step for brevity)}: for given data $(\totalstressT^{n-1},\pressureT^{n-1})\in \StressT \times \PressureT$ at the previous time step, seek $\answer{(\totalstressT,\pressureT,\displacementT,\vorticityT,\integratedvelocityT) \in {\bf X}_\tri}$ such that for all 
$(\teststressT,\testpressureT,
\testdisplacementT,
\testvorticityT,
\testvelocityT
 ) \in {\bf X}_\tri$ it holds
\begin{equation}
\label{eq:fully-discrete}
\begin{aligned}
\answer{(\mathbb{A} \totalstressT,\teststressT)}
 &+ \tilde \alpha ( \pressureT,\tr \teststressT)
  &+( \displacementT,\div \teststressT)
&+ (\vorticityT \cdot \chi,\teststressT)
 &&=0
 \\
 \tilde \alpha
\left(
\operatorname{tr} \totalstressT,\testpressureT
\right)&+
\answer{\tilde{c}_0 \left(\pressureT, \testpressureT\right)} &&&
+(\div \integratedvelocityT,\testpressureT)&= \langle \tilde f, \testpressureT \rangle
\\
(\div \totalstressT,\testdisplacementT)&&&&&=-({\bf g},\testdisplacementT)
\\
(\totalstressT,\testvorticityT\cdot \chi ) 
&&&&& = 0 
\\
 &
\ \ \ ( \pressureT,\div\testvelocityT)
&&&\answer{-\frac{1}{\Delta t}(\kappa^{-1}\integratedvelocityT,\testvelocityT)}
&=0.
\end{aligned}
\end{equation}
Similarly, in compact notation, as introduced in section~\ref{sec:porousmediamixed}, the discrete solutions $\tupleuT \in \tupleVT$ and $\tuplepT \in \tupleQT$ satisfy
\begin{equation}
\label{eq:discretegenericsaddlepoint}
\begin{aligned}
\mathcal A(\tupleuT, \tuplevT) +
\mathcal B(\tuplevT, \tuplepT) & =
\mathcal F( \tuplevT)
\\
\mathcal B(\tupleuT, \tupleqT)  - 
\mathcal C(\tuplepT, \tupleqT)& =
\mathcal G( \tupleqT)
\end{aligned}
\end{equation}
for all $(\tuplevT, \tupleqT) \in \tupleVT \times \tupleQT$.

The bilinear form $\mathcal B$ induces the discrete kernel spaces
\begin{align}
H_\tri 
&
:=
\{\tupleqT\in\tupleQT :
\mathcal B(\tuplevT,\tupleqT)=0
\ \forall \tuplevT\in\tupleVT\},
\\
K_\tri  
&
:=
\{\tuplevT\in\tupleVT :
\mathcal B(\tuplevT,\tupleqT)=0
\ \forall \tupleqT\in\tupleQT\}.
\end{align}
In order to mimic the structural properties of the \answer{semi-discrete} formulation,
we assume that the discrete spaces are chosen such that the kernel
structure of the bilinear form $\mathcal B$ is preserved at the discrete
level. 

\begin{lemma}
\label{lem:discrete_kernels_E1B}
Assume the discrete spaces satisfy 
$
\div(\VelocityT)\subseteq \PressureT 
$
as well as
\begin{itemize}
\item[(A)] There exists $\beta_u>0$, independent of the mesh-size, such that
\begin{equation}\label{eq:E1_disc_div_infsup}
\inf_{0\neq \displacementT\in \DisplacementT}
\sup_{0\neq \teststressT\in \StressT}
\frac{(\displacementT,\div\teststressT)}
{\|\teststressT\|_{H(\div)}\,\|\displacementT\|_0}
\ge \beta_u .
\end{equation}

\item[(B)]
There exists a finite-dimensional space $\PsiT \subset H^1_0(\Omega)^d$ (for $d=3$, see Remark~\ref{remark:B3d})
such that
\[
\curl \PsiT \subseteq \StressT ,
\]
and there exists $\beta_\xi>0$, independent of $h$, such that
\begin{equation}\label{eq:B_disc_stokes_infsup}
\inf_{0\neq \vorticityT\in \VorticityT}
\sup_{0\neq \psi_\tri\in \PsiT}
\frac{(\div\psi_\tri,\vorticityT)}
{\|\psi_\tri\|_{1}\,\|\vorticityT\|_0}
\ge \beta_\xi .
\end{equation}
\end{itemize}
Then:

\begin{subequations}
\begin{align}
H_\tri
&=
\{(\bf{0},\bf{0},\integratedvelocityT)\in\tupleQT:\ 
(\testpressureT,\div\integratedvelocityT)=0\ \forall \testpressureT\in\PressureT\},
\label{eq:HT_explicit}
\\ 
K_\tri
&=
\Big\{
(\teststressT,\testpressureT)\in\tupleVT :
(\testpressureT,\div\testvelocityT)=0 \ \forall \testvelocity\in\VelocityT,
(\displacementT,\div\teststressT)=0 \ \forall \displacementT\in\DisplacementT,
(\testvorticityT\cdot\chi,\teststressT)=0\  \forall \testvorticityT\in\VorticityT
\Big\}. \label{eq:KT_explicit}
\end{align}
\end{subequations}
\end{lemma}

\begin{proof}
Let $\tupleqT=(\displacementT,\vorticityT,\integratedvelocityT)\in H_\tri$.
By definition, $\mathcal B(\tuplevT,\tupleqT)=0$, for all $\tuplevT\in\tupleVT$.
Choosing $\tuplevT=(\underline{\bf{0}},\testpressureT)$ with arbitrary $\testpressureT\in\PressureT$ yields
\begin{equation}\label{eq:HT_test_pressure}
0=\mathcal B((\underline{\bf{0}},\testpressureT),(\displacementT,\vorticityT,\integratedvelocityT))=(\testpressureT,\div\integratedvelocityT)
\qquad \forall \testpressureT\in\PressureT .
\end{equation}
Since $\div\integratedvelocityT\in\div(\VelocityT)\subseteq \PressureT$, we may take
$\testpressureT=\div\integratedvelocityT$ in \eqref{eq:HT_test_pressure} and obtain
\[
\|\div\integratedvelocityT\|_0^2=(\div\integratedvelocityT,\div\integratedvelocityT)=0,
\]
hence $\div\integratedvelocityT=0$.
Next choose $\tuplevT=(\teststressT,0)$ with arbitrary $\teststressT\in\StressT$.
Then
\begin{equation}\label{eq:HT_constraint}
0=\mathcal B((\teststressT,0),(\displacementT,\vorticityT,\integratedvelocityT))
=(\displacementT,\div\teststressT)+(\vorticityT\cdot\chi,\teststressT)
\qquad \forall \teststressT\in\StressT .
\end{equation}
Pick $\teststressT=\curl\psi_\tri$ with $\psi_\tri\in\PsiT$. Using $\div(\curl\psi_\tri)=0$
elementwise gives
\[
0=(\displacementT,\div(\curl\psi_\tri))+(\vorticityT\cdot\chi,\curl\psi_\tri)
=(\div\psi_\tri,\vorticityT)\qquad \forall \psi_\tri\in\PsiT .
\]
By the discrete inf--sup \eqref{eq:B_disc_stokes_infsup} we conclude $\vorticityT=0$.
With $\vorticityT=0$, \eqref{eq:HT_constraint} reduces to
\[
(\displacementT,\div\teststressT)=0\qquad \forall \teststressT\in\StressT,
\]
and \eqref{eq:E1_disc_div_infsup} implies $\displacementT=0$.
Thus $\tupleqT=(\bf{0},\bf{0},\integratedvelocityT)$ and $\div\integratedvelocityT=0$, i.e.
\(
H_\tri\subseteq \{(\bf{0},\bf{0},\integratedvelocityT):\div\integratedvelocityT=0\}.
\)

Conversely, if $\tupleqT=(\bf{0},\bf{0},\integratedvelocityT)$ with $\div\integratedvelocityT=0$, then for any
$\tuplevT=(\teststressT,\testpressureT)\in\tupleVT$,
\[
\mathcal B(\tuplevT,\tupleqT)=(\testpressureT,\div\integratedvelocityT)=0,
\]
hence $\tupleqT\in H_\tri$ and \eqref{eq:HT_explicit} follows.

Now, let $\tuplevT=(\teststressT,\testpressureT)\in K_\tri$.
By definition,
$
\mathcal B(\tuplevT,\tupleqT)=0$, for all $\tupleqT=(\displacementT,\vorticityT,\integratedvelocityT)\in\tupleQT .
$
Testing with $\tupleqT=(\bf{0},\bf{0},\integratedvelocityT)$ for arbitrary $\integratedvelocityT\in\VelocityT$ gives
\[
0=\mathcal B((\teststressT,\testpressureT),(\bf{0},\bf{0},\integratedvelocityT))=(\testpressureT,\div\integratedvelocityT)
\qquad\forall \integratedvelocityT\in\VelocityT ,
\]
i.e. $(\testpressureT,\div\integratedvelocityT)=0$ for all $\integratedvelocityT\in\VelocityT$.
Testing with $\tupleqT=(\bf{0},\displacementT,\bf{0})$ for arbitrary $\displacementT\in\DisplacementT$ yields
\[
0=\mathcal B((\teststressT,\testpressureT),(\bf{0},\displacementT,\bf{0}))=(\displacementT,\div\teststressT)
\qquad\forall \displacementT\in\DisplacementT \ .
\]
Testing with $\tupleqT=(\bf{0},\testvorticityT \bf{0})$ for arbitrary $\testvorticityT\in\VorticityT$ gives
\[
0=\mathcal B((\teststressT,\testpressureT),(\bf{0},\testvorticityT, \bf{0}))=(\testvorticityT\cdot\chi,\teststressT)
\qquad\forall \testvorticityT\in\VorticityT,
\]
and every $(\teststressT,\testpressureT)\in K_\tri$ satisfies the three orthogonality conditions in \eqref{eq:KT_explicit}.
Conversely, let $(\teststressT,\testpressureT)\in\tupleVT$ satisfy
\[
(\testpressureT,\div\integratedvelocityT)=0 \quad \forall \integratedvelocityT\in\VelocityT,
\
(\displacementT,\div\teststressT)=0 \quad \forall \displacementT\in\DisplacementT
\text{ and }
(\testvorticityT\cdot\chi,\teststressT)=0 \quad \forall \testvorticityT\in\VorticityT.
\]
Then for every $\tupleqT=(\displacementT,\vorticityT, \integratedvelocityT)\in\tupleQT$,
\[
\mathcal B((\teststressT,\testpressureT),\tupleqT)
=
(\testpressureT,\div\integratedvelocityT)
+(\displacementT,\div\teststressT)
+(\vorticityT\cdot\chi,\teststressT)
=0.
\]
Thus $(\teststressT,\testpressureT)\in K_\tri$, and \eqref{eq:KT_explicit} follows.
\end{proof}

\begin{remark}[Assumption (B) for $d=3$\label{remark:B3d}]
Assumption \textbf{(B)} is formulated such that it is directly meaningful for $d=2$. For $d\in\{2,3\}$, it is to be read with $\PsiT\subset H^1_0(\Omega)^{d\times(2d-3)}$, the curl being taken row-wise, i.e., $(\curl\psi)_{ij}=\varepsilon_{jkl}\,\partial_k\psi_{il}$ for $d=3$, and with the pairing $(\div\psi_\tri,\vorticityT)$ in \eqref{eq:B_disc_stokes_infsup} replaced by $(\curl\psi_\tri,\vorticityT\cdot\chi)$; for $d=2$ the two pairings coincide up to sign. With this reading, the proof of Lemma~\ref{lem:discrete_kernels_E1B} applies verbatim. For $d=3$, it holds, up to sign, $(\curl\psi,\vorticity\cdot\chi)=(\div\tilde\psi,\vorticity)$ with $\tilde\psi:=\psi^\top-(\operatorname{tr}\psi)\mathbf{I}$ and the row-wise divergence. Since $\psi=\tilde\psi^\top-\frac12(\operatorname{tr}\tilde\psi)\mathbf{I}$ can be recovered from $\tilde\psi$, and this map preserves $H^1_0(\Omega)^{3\times 3}$ as well as componentwise polynomial spaces, assumption \textbf{(B)} amounts to three Stokes inf--sup conditions, one for each component of $\vorticityT$. Accordingly, the pairs $(\Psi_r(T),\Xi_r(T))$ in Sections~4.1--4.2 are to be understood as $2d-3$ copies of the velocity and pressure spaces of the respective Stokes-stable pair, e.g., $(\mathbb P_{r+1}(T)^{3\times 3},\mathbb P_r(T)^{3})$ for Taylor--Hood elements with $d=3$; the inclusion conditions on $\curl\PsiT$ stated there remain valid, as the curl acts row-wise.
\end{remark}

\begin{remark}
If, in addition, $\div(\VelocityT)=\PressureT$, then the first orthogonality
condition in \eqref{eq:KT_explicit} implies $\testpressureT=0$, and therefore
\[
K_\tri
=
\Big\{
(\teststressT,0)\in\tupleVT :
(\displacementT,\div\teststressT)=0 \ \forall \displacementT\in\DisplacementT,\ 
(\testvorticityT\cdot\chi,\teststressT)=0\ \forall \testvorticityT\in\VorticityT
\Big\}.
\]
\end{remark}

By Lemma~\ref{lem:discrete_kernels_E1B}, the discrete kernels
$H_\tri$ and $K_\tri$ satisfy the same structural constraints
as their continuous counterparts.
On $H_\tri$, any element has the form $(\bf{0},\bf{0},\integratedvelocityT)$ with
$\div\integratedvelocityT=0$. Hence, by uniform ellipticity of
$\answer{\kappa^{-1}}$,
\[\frac{1}{\Delta t}(\answer{\kappa^{-1}}\testvelocityT,\testvelocityT)
\ge  
\frac{1}{\answer{\overline\kappa}\,\Delta t}
\|\testvelocityT\|_{H(\div)}^2
\qquad \forall (\bf{0}, \bf{0}, \testvelocityT)\in H_\tri .
\]
On $K_\tri$, the divergence-free condition and the discrete weak
symmetry constraint eliminate the nullspace of the compliance operator.
Therefore, the uniform positivity of $\mathcal A$ implies 
\[
\mathcal A((\teststressT,\testpressureT),(\teststressT,\testpressureT))
\ge 
\answer{\left(2 \mu C_K^2\right)^{-1}}
\big(
\|\teststressT\|_{H(\div)}^2+\|\testpressureT\|_0^2
\big)
\qquad \forall (\teststressT,\testpressureT)\in K_\tri .
\]
With coercivity established, the well-posedness of the discrete problem 
\eqref{eq:discretegenericsaddlepoint} now depends solely on the 
satisfaction of the discrete inf-sup condition for the bilinear form $\mathcal{B}$, which we analyze in the following lemma.
\begin{lemma}
\label{lem:disc_infsup_Hperp}
Assume that $\div(\VelocityT)\subseteq \PressureT$ and that the
discrete elasticity stability conditions \textbf{(A)}--\textbf{(B)}
of Lemma~\ref{lem:discrete_kernels_E1B} hold.
Then there exists $\beta_\tri>0$, independent of the mesh-size, such that
\begin{equation}\label{eq:disc_infsup}
\inf_{0\neq \tupleqT\in H_\tri^\perp}\ 
\sup_{0\neq \tuplevT\in\tupleVT}
\frac{\mathcal B(\tuplevT,\tupleqT)}
{\|\tuplevT\|_{\tupleV}\,\|\tupleqT\|_{\tupleQ}}
\ \ge\ \beta_\tri 
\end{equation}
and the discrete problem
\eqref{eq:discretegenericsaddlepoint} admits a unique solution
$(\tupleuT,\tuplepT)\in\tupleVT\times\tupleQT$.
Moreover, the stability estimate
\[
\|\tupleuT\|_{\tupleV}+\|\tuplepT\|_{\tupleQ}
\ \le\
C\Big(\|\mathcal F\|_{\tupleV'}+\|\mathcal G\|_{\tupleQ'}\Big)
\]
holds with a constant $C>0$ independent of the mesh-size.
\end{lemma}

\begin{proof}
Let $\tupleqT=(\displacementT,\vorticityT,\integratedvelocityT)\in H_\tri^\perp$ be arbitrary. By definition of $H_\tri^\perp$, the velocity component $\integratedvelocityT\in\VelocityT$
satisfies the orthogonality condition defining $H_\tri^\perp$. Hence
\begin{equation}\label{eq:vel_equiv_on_Hperp}
\|\integratedvelocityT\|_{H(\div)} \eqsim \|\div\integratedvelocityT\|_0
\qquad\text{for }\integratedvelocityT\in H_\tri^\perp.
\end{equation}
Choose
\[
q_* := \div\integratedvelocityT \in \PressureT ,
\]
which is admissible since $\div(\VelocityT)\subseteq \PressureT$.
Next, apply the assumptions (A) and (B) from Lemma~\ref{lem:discrete_kernels_E1B} 
to obtain a stress test function $\teststress_*\in\StressT$ such that
\begin{equation}\label{eq:tau_star_props}
(\displacementT,\div\teststress_*)+(\vorticityT\cdot\chi,\teststress_*)
\ \gtrsim\ \|\displacementT\|_0^2+\|\vorticityT\|_0^2,
\qquad
\|\teststress_*\|_{H(\div)}\ \lesssim\ \|\displacementT\|_0+\|\vorticityT\|_0 .
\end{equation}
Define $\tuplev_*:=(\teststress_*,q_*)\in\tupleVT$.
Then, by definition of $\mathcal B$,
\[
\mathcal B(\tuplev_*,\tupleqT)
=
(q_*,\div\integratedvelocityT)
+(\displacementT,\div\teststress_*)
+(\vorticityT\cdot\chi,\teststress_*).
\]
Using $q_*=\div\integratedvelocityT$ and \eqref{eq:tau_star_props},
\begin{equation}\label{eq:B_lower_bound}
\mathcal B(\tuplev_*,\tupleqT)
\ \gtrsim\
\|\div\integratedvelocityT\|_0^2+\|\displacementT\|_0^2+\|\vorticityT\|_0^2 .
\end{equation}
Moreover,
\[
\|\tuplev_*\|_{\tupleV}
\lesssim
\|\teststress_*\|_{H(\div)}+\|q_*\|_0
\lesssim
\|\displacementT\|_0+\|\vorticityT\|_0+\|\div\integratedvelocityT\|_0 .
\]
Combining with \eqref{eq:vel_equiv_on_Hperp} gives
\[
\|\tupleqT\|_{\tupleQ}
=
\big(\|\integratedvelocityT\|_{H(\div)}^2+\|\displacementT\|_0^2+\|\vorticityT\|_0^2\big)^{1/2}
\ \eqsim\
\big(\|\div\integratedvelocityT\|_0^2+\|\displacementT\|_0^2+\|\vorticityT\|_0^2\big)^{1/2}.
\]
Therefore, \eqref{eq:B_lower_bound} implies
\[
\frac{\mathcal B(\tuplev_*,\tupleqT)}
{\|\tuplev_*\|_{\tupleV}\,\|\tupleqT\|_{\tupleQ}}
\ \gtrsim\ 1,
\]
with constants independent of the mesh-size. Taking the supremum over $\tuplevT\in\tupleVT$
and then the infimum over $\tupleqT\in H_\tri^\perp\setminus\{0\}$ yields
\eqref{eq:disc_infsup}.
 \answer{The constant $C$ has the same structure as in Theorem~\ref{thm:wellposedness}.}
\end{proof}

\answer{
In contrast to the evolution-type error analysis of Lee~\cite{Lee:2016}, the
preceding analysis treats each time step as a perturbed saddle-point problem
with fully explicit stability constants, whose explicit form is the basis for
the convergence analysis of the tuned splitting scheme in Section~5. As a
by-product, we} can now list standard finite element families that satisfy
assumptions \textbf{(A)}--\textbf{(B)} of
Lemma~\ref{lem:discrete_kernels_E1B}
as well as the Darcy compatibility condition
$\div(\VelocityT)\subseteq \PressureT$.
We distinguish between simplicial and tensor-product meshes. 

\subsection{Finite elements on simplicial meshes}

We start with $H(\div)$-conforming finite element spaces on simplices
and define the global velocity and stress spaces by
\[
\VelocityT
:=
\left\{
\testvelocity \in \hdiv :
\ \testvelocity|_T \in \mathbb Y_\ell(T)
\ \forall T\in\tri
\right\},
\text{ and }
\StressT
:=
\left\{
\teststress \in \hdiv^d :
\ (\teststress_i)|_T \in \mathbb X_k(T)
\ \forall T\in\tri,\ i=1,\dots,d
\right\},
\]
where the local spaces $\mathbb X_k(T)$ and $\mathbb Y_\ell(T)$ are
chosen from the classical families
\[
\mathbb X_k(T),\ \mathbb Y_\ell(T)
\in
\{\RT_k(T),\BDM_k(T),\BDFM_k(T),\ABF_k(T)\},
\]
with
\begin{align*}
\RT_k(T)
&:= \mathbb P_k(T)^d + \mathbf{x}\,\mathbb P_k(T),
\\
\BDM_k(T)
&:= \mathbb P_k(T)^d,
\qquad k\ge1,
\\
\ABF_k(T)
&:= \mathbb P_k(T)^d + \mathbf{x}\,\mathbb P_k(T)^d,
\qquad k\ge1,
\\
\BDFM_k(T)
&:= \RT_k(T)\oplus \bigl(b_T\,\mathbb P_{k-1}(T)^d\bigr),
\qquad k\ge1.
\end{align*}
Here $\mathbf{x}$ denotes the identity map on $T$, and
\[
b_T=\prod_{i=1}^{d+1}\lambda_i
\]
is the element bubble function defined in terms of the barycentric
coordinates $\lambda_i$ of $T$.

The scalar spaces $\DisplacementT$ and $\PressureT$ are then chosen
according to the divergence constraints appearing in
Lemma~\ref{lem:discrete_kernels_E1B}. More precisely, the displacement
space $\DisplacementT$ is selected so that assumption \textbf{(A)} holds,
while the pressure space $\PressureT$ is chosen such that
\[
\div(\VelocityT)\subseteq \PressureT.
\]
For the classical $H(\div)$-conforming families introduced above, the
divergence operator is surjective onto the corresponding polynomial
spaces, i.e.,
\[
\div \RT_k(T)=\mathbb P_k(T),
\qquad
\div \BDM_k(T)=\div \BDFM_k(T)=\div \ABF_k(T)=\mathbb P_{k-1}(T).
\]
Hence, the minimal compatible choices for $\DisplacementT$ and
$\PressureT$ coincide with the discontinuous polynomial spaces
\[
\DisplacementT=
\mathbb P_k(\tri)^d
\quad\text{if }\StressT\text{ is based on }\RT_k,
\qquad
\DisplacementT=
\mathbb P_{k-1}(\tri)^d
\quad\text{if }\StressT\text{ is based on }
\BDM_k,\ \BDFM_k,\ \text{or }\ABF_k,
\]
and similarly
\[
\PressureT=
\mathbb P_\ell(\tri)
\quad\text{if }\VelocityT\text{ is based on }\RT_\ell,
\qquad
\PressureT=
\mathbb P_{\ell-1}(\tri)
\quad\text{if }\VelocityT\text{ is based on }
\BDM_\ell,\ \BDFM_\ell,\ \text{or }\ABF_\ell.
\]
It remains to choose the spaces $\PsiT$ and $\VorticityT$ appearing in
assumption \textbf{(B)}. To this end we consider the Stokes-stable pairs
\[
\PsiT :=
\{\psi\in H^1_0(\Omega)^d:\ \psi|_T\in\Psi_r(T)\ \forall T\in\tri\},
\qquad
\VorticityT :=
\{\vorticity\in L^2(\Omega)^{2d-3}:\ \vorticity|_T\in\Xi_r(T)\ \forall T\in\tri\}
\]
with the following possibilities (for $d=3$, see Remark~\ref{remark:B3d}).
\begin{itemize}

\item Taylor--Hood elements $
(\Psi_r(T),\Xi_r(T))
=
(\mathbb P_{r+1}(T)^d,\mathbb P_r(T)^{2d-3})$
for $r\ge1$. Since
$
\curl \Psi_r(\tri)\subseteq \mathbb P_r(\tri)^d,
$
the inclusion $\curl\PsiT\subseteq\StressT$ holds provided that $r\le k$. Here, the $\BDM$ case coincides with the
Arnold--Falk--Winther choice.
\item Mini elements $
(\Psi(T),\Xi(T))
=
(\mathbb P_1(T)^d \oplus B_{d+1}(T)^d,\mathbb P_1(T)^{2d-3}),
$
where $B_{d+1}(T)$ denotes the element bubble  of degree $d+1$.
Since
$
\curl \Psi(\tri)\subseteq \mathbb P_d(\tri)^d,
$
the inclusion $\curl\PsiT\subseteq\StressT$ holds provided that $d\le k$.

\item {Scott--Vogelius elements}
$
(\Psi_r(T),\Xi_r(T))
=
(\mathbb P_r(T)^d,\mathbb P_{r-1}(T)^{2d-3}),
$ for $r\ge 2d$ under 
additional mesh conditions.
Since
$
\curl \Psi_r(\tri)\subseteq \mathbb P_{r-1}(\tri)^d,
$
the inclusion $\curl\PsiT\subseteq\StressT$ holds provided that $r\le k+1$.
\end{itemize}

\subsection{Finite elements on tensor-product meshes}

We now consider tensor-product meshes and let $\tri$ be a shape-regular
partition of $\Omega$ into quadrilaterals for $d=2$ or hexahedra for
$d=3$. On such meshes, let $\mathbb Q_{\mathbf r}(T)$ denote the tensor-product polynomial
space on $T$ with multi-degree
$\mathbf r=(r_1,\dots,r_d)$, i.e.
\[
\mathbb Q_{\mathbf r}(T)
:=
\operatorname{span}\{
x_1^{\alpha_1}\cdots x_d^{\alpha_d}:
0\le \alpha_i\le r_i,\ i=1,\dots,d
\}.
\]
In particular, we write
$\mathbb Q_k(T):=\mathbb Q_{(k,\dots,k)}(T)$.
As before, we define the global velocity and stress spaces by
\[
\VelocityT :=
\left\{
\testvelocity \in \hdiv :
\ \testvelocity|_T \in \mathbb Y_\ell(T)
\ \forall T\in\tri
\right\},
\qquad
\StressT :=
\left\{
\teststress \in \hdiv^d :
\ (\teststress_i)|_T \in \mathbb X_k(T)
\ \forall T\in\tri,\ i=1,\dots,d
\right\}.
\]
The local spaces $\mathbb X_k(T)$ and $\mathbb Y_\ell(T)$ are chosen
from the classical tensor-product $H(\div)$ families
\[
\mathbb X_k(T),\ \mathbb Y_\ell(T)
\in
\{\RT_k(T),\BDDF_k(T)\}.
\]
with
\[
\RT_k(T)
=
\begin{cases}
\mathbb Q_{k+1,k}(T)\times \mathbb Q_{k,k+1}(T), & d=2,
\\
\mathbb Q_{k+1,k,k}(T)\times
\mathbb Q_{k,k+1,k}(T)\times
\mathbb Q_{k,k,k+1}(T),
& d=3,
\end{cases}\text{   and }
\BDDF_k(T)
:=
\mathbb Q_k(T)^d,
\qquad k\ge1.
\]
For the classical tensor-product $H(\div)$-conforming families
introduced above, the divergence operator maps onto the
corresponding scalar polynomial spaces, i.e. 
$
\div \RT_k(T)=\mathbb Q_k(T),
$ and $
\div \BDDF_k(T)=\mathbb Q_{k-1}(T).
$
Consequently, the minimal compatible scalar spaces are
\[
\DisplacementT=
\mathbb Q_k(\tri)^d
\quad\text{if }\StressT\text{ is based on }\RT_k,
\qquad
\DisplacementT=
\mathbb Q_{k-1}(\tri)^d
\quad\text{if }\StressT\text{ is based on }\BDDF_k,
\]
and similarly
\[
\PressureT=
\mathbb Q_\ell(\tri)
\quad\text{if }\VelocityT\text{ is based on }\RT_\ell,
\qquad
\PressureT=
\mathbb Q_{\ell-1}(\tri)
\quad\text{if }\VelocityT\text{ is based on }\BDDF_\ell.
\]
It remains to choose the spaces $\PsiT$ and $\VorticityT$ appearing in
assumption \textbf{(B)}. To this end we consider the Stokes-stable pairs
\[
\PsiT :=
\{\psi\in H^1_0(\Omega)^d:\ \psi|_T\in\Psi_r(T)\ \forall T\in\tri\},
\qquad
\VorticityT :=
\{\vorticity\in L^2(\Omega)^{2d-3}:\ \vorticity|_T\in\Xi_r(T)\ \forall T\in\tri\}.
\]
A natural tensor-product analogue is given by the Taylor--Hood family
\[
(\Psi_r(T),\Xi_r(T))
=
(\mathbb Q_{r+1}(T)^d,\mathbb Q_r(T)^{2d-3}),
\qquad r\ge1.
\]
Since $
\curl \Psi_r(\tri)\subseteq \mathbb Q_r(\tri)^d$, 
the inclusion $\curl\PsiT\subseteq\StressT$ holds provided that
\[
r\le k \quad\text{if }\mathbb X_k(T)=\RT_k(T),
\qquad
r\le k-1 \quad\text{if }\mathbb X_k(T)=\BDDF_k(T).
\]
Another possibility is the Mini element $
(\Psi(T),\Xi(T))
=
(\mathbb Q_1(T)^d \oplus B(T)^d,\mathbb Q_1(T)^{2d-3})
$, 
where $B(T)=\prod_{i=1}^{d}(1-x_i^2)$ denotes the tensor-product bubble function.
Since
$
\curl \Psi(\tri)\subseteq \mathbb Q_1(\tri)^d
$, 
the inclusion $\curl\PsiT\subseteq\StressT$ holds provided that
\[
1\le k \quad\text{if }\mathbb X_k(T)=\RT_k(T),
\qquad
2\le k \quad\text{if }\mathbb X_k(T)=\BDDF_k(T).
\]

\section{Iterative coupling for the fully-discrete system}
\label{sec:iterativecoupling}
Our aim is to design an iterative coupling scheme for the fully-discrete formulation of the Biot equations~\eqref{eq:fully-discrete}, decoupling flow and mechanics; the scheme and discussion will be independent of the particular spatial discretization presented in Section~\ref{sec:fem}. \answer{As discussed before, due to the use of a stress-based formulation}, the coupling between the flow and mechanics subproblems is in fact symmetric. \answer{This is tightly connected to the above highlighted minimization character of the five-field formulation. The coupling across flow and deformation occurs solely through a quadratic energy, while constraints act merely onto the single subphysics. As the partial differential equations are associated to suitable derivatives of the energy, the flow-deformation coupling in form of stress-pressure coupling terms is finally symmetric}; for a more detailed discussion on the inherent gradient flow structure and resulting symmetries, we refer to~\cite{both2019gradient}. Symmetric couplings can be decoupled in iterative fashion by a simple alternating minimization approach, and its convergence follows from abstract results solely relying on convexity and continuity properties~\cite{both2022rate}. As discussed in~\cite{both2019gradient} alternating minimization can be identified with the common fixed-stress split, cf., e.g.,~\cite{kim2011stability}, \answer{requiring reformulation and reinterpretation of primary variables when converting to a two-field formulation}.

Inspired by previous problem-specific developments on tuned stabilization of the fixed-stress split, e.g.~\cite{mikelic2013convergence,BothBorregalesNordbottenKumarRadu:2017,storvik2019optimization}, we investigate the possibility of using stabilization, introducing the possibility for tuning the convergence speed. However, due to the symmetric coupling, instead of adding \answer{positive} stabilization, an improved approximation of the Schur complement is in fact expected by applying negative stabilization, i.e., seemingly destabilization. This stands in contrast to previous attempts to accelerate splitting of symmetrically coupled problems, cf., e.g.,~\cite{nuca2024splitting}, while a similar conclusion has been arrived in~\cite{both2022iterative} for a different poromechanics model with symmetric coupling between displacements and fluxes. 

\paragraph{Fixed-stress split} Finally, resembling the same approach as the fixed-stress split, we consider a two-step iterative coupling scheme with $i\geq 1$, denoting the iteration index \answer{and consistently applied as superscript indicating the iteration}. For simplicity, we assume initialization of the scheme at $i=0$ through the use of the data at the previous time step. Then the following iteration is repeated until convergence in any user-defined norm.

\paragraph{Step 1: Flow}
Let $\totalstressT^{i-1}\in \StressT$ be given, seek $(\pressureT^{i},\integratedvelocityT^{i}) \in \PressureT \times \VelocityT$ such that for all 
$(\testpressureT,
\testvelocityT) \in \PressureT \times \VelocityT$ it holds
\begin{equation}
\label{eq:splitting-flow}
\begin{aligned}
 &
\tilde{c}_0\left(\pressureT^{i},\testpressureT\right) - \beta\left(\pressureT^{i} - \pressureT^{i-1},\testpressureT\right) &&&
+ (\div \integratedvelocityT^{i},\testpressureT)&= \langle \tilde{f},\testpressureT \rangle - \tilde \alpha
\left(
\operatorname{tr} \totalstressT^{i-1},\testpressureT
\right)
\\
&
(\pressureT^{i},\div\testvelocityT)
&&&-\frac{1}{\Delta t}(\kappa^{-1} \integratedvelocityT^{i}, \testvelocityT)
&=0.
\end{aligned}
\end{equation}
\paragraph{Step 2: Mechanics} Seek $(\totalstressT^{i},\displacementT^{i},\vorticityT^{i}) \in \StressT \times \DisplacementT \times \VorticityT$ such that for all 
$(\teststressT,
\testdisplacementT,
\testvorticityT
 ) \in \StressT \times \DisplacementT \times \VorticityT$ it holds
\begin{equation}
\label{eq:splitting-mechanics}
\begin{aligned}
(\mathbb{A} \totalstressT^{i},\teststressT)
&+(\displacementT^{i},\div \teststressT)
&+(\vorticityT^{i} \cdot \chi,\teststressT)
&=-{\tilde \alpha } ( \pressureT^{i},\tr(\teststressT))
\\
(\div \totalstressT^{i},\testdisplacementT)&&&=-({\bf g},\testdisplacementT)
\\
(\totalstressT^{i},\testvorticityT\cdot \chi ) 
&&&= 0.
\end{aligned}
\end{equation}

\begin{remark}[Standard fixed-stress split]
The choice $\beta=0$ resembles the standard fixed-stress split, iterating between solving the flow problem under fixed stress conditions, and updating the mechanics variables subsequently. It is unconditionally convergent for the symmetric problem~\cite{both2019gradient,both2022rate}.
\end{remark}

\begin{remark}[Tuned fixed-stress split]
\answer{It is expected that the tuned fixed-stress split converges for any $\beta\leq 0$, yet with worse convergence rate than for the standard fixed-stress split. Acceleration is possible for $\beta >0$, yet there is an upper bound until which convergence can be guaranteed. This upper bound, associated to maximal stabilization is associated to the coercivity of the Schur complement. The following Lemma investigates this upper bound.}
\end{remark}

We employ a problem-specific analysis to identify suitable values for $\beta\geq 0$ with optimized convergence rate in terms of the fluid pressure, following mostly the techniques in~\cite{BothBorregalesNordbottenKumarRadu:2017}.

\begin{lemma}[Convergence of a tuned fixed-stress split\label{lemma:splitting}]
Let $i\geq 1$ and $(\totalstressT,\pressureT,\displacementT,\vorticityT,\integratedvelocityT) \in \bf{X}_\tri$ and $(\totalstressT^{i},\pressureT^{i},\displacementT^{i},\vorticityT^{i},\integratedvelocityT^{i}) \in \bf{X}_\tri$ denote the fully-discrete and the iterative solutions due to~\eqref{eq:fully-discrete} and~\eqref{eq:splitting-flow}--\eqref{eq:splitting-mechanics}, respectively. Let
\begin{align}
\label{eq:error-quantities}
e_{\totalstress}^i := \totalstressT^{i} - \totalstressT,\ \  
e_p^i := \pressureT^{i} - \pressureT,\ \ 
e_{\displacement}^i := \displacementT^{i} - \displacementT,\ \ 
e_{\vorticity}^i := \vorticityT^{i} - \vorticityT,\ \ 
e_{\integratedvelocity}^i := \integratedvelocityT^{i} - \integratedvelocityT,
\end{align}
denote the respective error quantities. Then for $\beta \leq \frac{1}{2}\frac{d\alpha^2}{2\mu+d\lambda}$ it holds
\begin{align*}
\left(2 c_0 + 2C_\mathrm{P}^2\underline{\kappa}\Delta t + \frac{d\alpha^2}{2\mu+d\lambda} - \beta \right) \| e_p^{i}\|^2 
\leq \left(\frac{d\alpha^2}{2\mu+d\lambda} - \beta \right) \|e_p^{i-1}\|^2,
\end{align*}
where $C_\mathrm{P}$ denotes a Poincar\'e constant, and $\underline{\kappa}$ is the lower bound of the permeability $\kappa$. Thus, the fixed-stress split ($\beta=0$) converges, and the predicted convergence is optimized for the destabilization parameter $\beta_\mathrm{max} :=
\frac{1}{2}\frac{d\alpha^2}{2\mu+d\lambda}$. 
\end{lemma}
\begin{proof}
Let $\beta \geq 0$ (to be specified later). Taking the difference between~\eqref{eq:splitting-flow}--\eqref{eq:splitting-mechanics} and~\eqref{eq:fully-discrete}, the error quantity $(e_{\totalstress}^i, e_p^i, e_{\displacement}^i, e_{\vorticity}^i, e_{\integratedvelocity}^i) \in {\bf X}_\tri$ satisfies the error equations
\begin{subequations}
\label{eq:split-convergence-proof-1}
\begin{align}
\label{eq:split-convergence-proof-1a}
(\answer{\mathbb{A}}e_{\totalstress}^{i},\teststressT)
&+ {\tilde \alpha } ( e_p^{i},\tr\teststressT)
&+( e_\mathbf{u}^{i},\div \teststressT)
&+ (e_{\boldsymbol{\xi}}^{i} \cdot \chi,\teststressT)
&&=0
\\
\label{eq:split-convergence-proof-1b}
\tilde \alpha
\left(
\operatorname{tr} e_{\totalstress}^{i-1},\testpressureT
\right)&+
\answer{\tilde{c}_0}\left(e_p^{i},\testpressureT\right) - \beta \left(e_p^{i} - e_p^{i-1}, \testpressureT\right)&&&
+(\div e_\integratedvelocity^{i},\testpressureT)&=0
\\
\label{eq:split-convergence-proof-1c}
(\div e_{\totalstress}^{i},\testdisplacementT)&&&&&= 0
\\
\label{eq:split-convergence-proof-1d}
(e_{\totalstress}^{i},\testvorticityT\cdot \chi ) 
&&&&& = 0 
\\
\label{eq:split-convergence-proof-1e}
 &
\ \ \ ( e_p^{i},\div\testvelocityT)
&&&\answer{-\frac{1}{\Delta t}(\kappa^{-1}e_\integratedvelocity^{i},\testvelocityT)}
&=0  
\end{align}
\end{subequations}
for all $(\teststressT,\testpressureT,\testdisplacementT,\testvorticityT,\testvelocityT) \in {\bf X}_\tri$.
%
To eliminate all coupling terms in~\eqref{eq:split-convergence-proof-1}, 
we test~\eqref{eq:split-convergence-proof-1a} with $\teststressT = -e_{\totalstress}^{i-1}$,
~\eqref{eq:split-convergence-proof-1b} with $\testpressureT = e_p^i$,
~\eqref{eq:split-convergence-proof-1c} at iteration $i-1$ with $\testdisplacementT = e_{\displacement}^i$,
~\eqref{eq:split-convergence-proof-1d} at iteration $i-1$ with $\testvorticityT = e_{\vorticity}^i$, and~\eqref{eq:split-convergence-proof-1e} with $\testvelocityT = - e_{\integratedvelocity}^i$ and obtain for their sum
\begin{equation}
\label{eq:split-convergence-proof-1-summed}
\begin{aligned}
&-(\mathbb{A} e_{\totalstress}^{i},e_{\totalstress}^{i-1})
+ \tilde{c}_0 \left(e_p^{i},e_p^i\right)
+\frac{1}{\Delta t}(\kappa^{-1} e_\integratedvelocity^{i},e_{\integratedvelocity}^i)
- \beta \left(e_p^{i} - e_p^{i-1}, e_p^i\right)
=0.
\end{aligned}
\end{equation}
To simplify the expression, we reduce the flux term to a pressure term by employing a Poincar\'e inequality-type argument for mixed flow problems, cf., e.g.,~\cite{BothBorregalesNordbottenKumarRadu:2017} in the context of poromechanics, introducing a Poincar\'e constant $C_\mathrm{P}$. We obtain
\begin{align}
     \tilde{c}_0 (e_p^i, e_p^i) + \frac{1}{\Delta t}\left(\kappa^{-1} e_{\integratedvelocity}^i, e_{\integratedvelocity}^i\right) \geq 
     \left(\tilde c_0 + C_\mathrm{P}^2\underline{\kappa}\Delta t \right) \|e_p^i\|^2 
     =: \left(\frac{d\alpha^2}{2\mu+d\lambda} + C\right) \|e_p^i\|^2,
\end{align}
where for brevity we introduced the effective constant $C := C(\mu, \lambda, c_0, \underline{\kappa}, C_\mathrm{P}, \Delta t)$. We split and reorganize the pressure terms, such that in summary, it holds
\begin{equation}
\label{eq:split-convergence-proof-1-summed-post}
\begin{aligned}
&C\| e_p^{i}\|^2 
+ \left(\frac{d\alpha^2}{2\mu+d\lambda} - \beta \right) \left(e_p^{i} - e_p^{i-1}, e_p^i\right) 
+ \frac{d\alpha^2}{2\mu+d\lambda}  \left(e_p^{i}, e_p^{i-1}\right) 
- (\mathbb{A} e_{\totalstress}^{i},e_{\totalstress}^{i-1})
\leq 0.
\end{aligned}
\end{equation}
We aim to make use of the binomial and polarization identities 
\begin{align}
\label{eq:binomial-identity}
    (a-b)a &= \frac{1}{2}a^2 + \frac{1}{2} (a-b)^2 - \frac{1}{2}b^2, \qquad 
    ab 
    = \frac{1}{4}(a+b)^2 - \frac{1}{4}(a-b)^2 
\end{align}
for placeholders $a$ and $b$, which also extent to general bilinear forms.
Application to~\eqref{eq:split-convergence-proof-1-summed-post} yields
\begin{equation}
\label{eq:split-convergence-proof-with-algebraic-identities}
\begin{aligned}
&\left(C + \frac{1}{2} \left(\frac{d\alpha^2}{2\mu+d\lambda} - \beta \right)\right) \| e_p^{i}\|^2 
+ \frac{1}{4}\left(\frac{d\alpha^2}{2\mu+d\lambda} - 2\beta \right) \left\|e_p^{i} - e_p^{i-1} \right\|^2
+ \frac{1}{4}(\mathbb{A} e_{\totalstress}^{i} - e_{\totalstress}^{i-1}, e_{\totalstress}^{i} - e_{\totalstress}^{i-1})\\
&+ \frac{1}{4}\frac{d\alpha^2}{2\mu+d\lambda}  \left\| e_p^{i} +  e_p^{i-1}\right\|^2
-\frac{1}{4}(\mathbb{A} e_{\totalstress}^{i} + e_{\totalstress}^{i-1}, e_{\totalstress}^{i} + e_{\totalstress}^{i-1})\\
&\leq \frac{1}{2} \left(\frac{d\alpha^2}{2\mu+d\lambda} - \beta \right) \|e_p^{i-1}\|^2.
\end{aligned}
\end{equation}
From~\eqref{eq:split-convergence-proof-1a}, summed for iterations $i$ and $i-1$ and tested with $\teststressT = e_{\totalstress}^i + e_{\totalstress}^{i-1}$, utilizing~\eqref{eq:split-convergence-proof-1c}--\eqref{eq:split-convergence-proof-1c} to drop coupling terms, and the definition of $\tilde{\alpha}= \frac{\alpha}{2\mu+d\lambda}$, we obtain
\begin{align}
\label{eq:cauchy-schwarz}
(\mathbb{A} e_{\totalstress}^{i} + e_{\totalstress}^{i-1}, e_{\totalstress}^{i} + e_{\totalstress}^{i-1})
=
-\tilde{\alpha}\left(e_p^i + e_p^{i-1}, \tr (e_{\totalstress}^i + e_{\totalstress}^{i-1}) \right)
\leq 
\frac{\sqrt{d}\,\alpha}{\sqrt{2\mu+d\lambda}}\,\left\|e_p^i + e_p^{i-1}\right\|\, \left(\mathbb{A} e_{\totalstress}^i + e_{\totalstress}^{i-1}, e_{\totalstress}^i + e_{\totalstress}^{i-1}\right)^{1/2},
\end{align}
where the inequality follows from the Cauchy-Schwarz inequality and a standard (pointwise) AM-QM inequality
\begin{align}
\label{eq:discussion-term-2a}
    (\mathbb{A} \teststressT, \teststressT) 
    =
    \frac{1}{2\mu} \left(\teststressT, \teststressT \right) - \frac{1}{2\mu}\, \frac{\lambda}{2\mu + d\lambda} \left(\tr\teststressT, \tr\teststressT\right)
    \geq 
    \frac{1}{d} \, \frac{1}{2\mu + d\lambda} \left(\tr \teststressT, \tr \teststressT \right).
\end{align}
Finally, under the condition $\beta \leq \frac{1}{2} \frac{d\alpha^2}{2\mu+d\lambda}$, we can drop various non-negative terms in~\eqref{eq:split-convergence-proof-with-algebraic-identities} and it remains
\begin{align*}
\left(C + \frac{1}{2} \left(\frac{d\alpha^2}{2\mu+d\lambda} - \beta \right)\right) \| e_p^{i}\|^2 
\leq \frac{1}{2} \left(\frac{d\alpha^2}{2\mu+d\lambda} - \beta \right) \|e_p^{i-1}\|^2.
\end{align*}
This concludes the proof.
\end{proof}

\begin{remark}[Convergence of all five fields]
For $\beta \leq \frac{1}{2}\frac{d\alpha^2}{2\mu + d\lambda}$, an argument based on an inequality of the type~\eqref{eq:cauchy-schwarz} results in direct convergence of the stress error $e_{\totalstress}^i$. Inf-sup stability of the mechanics subproblem results in convergence of the errors $(e_{\displacement}^i,e_{\vorticity}^i)$, while the overall inf-sup stability of the fully-mixed formulation, cf. Section~\ref{sec:porousmediamixed}, yields convergence of the flux error $e_{\integratedvelocity}^i$.
\end{remark}

\begin{remark}[Theoretically vs.\ practically optimal tuning\label{remark:optimality}]
As discussed in other works on optimizing the stabilization~\cite{both2017numerical,storvik2019optimization,storvik2020fixed}, the practically optimal tuning depends on further factors including separation of boundaries into Dirichlet and Neumann boundaries as well as the overall physical character of the solution, stability constants as inf-sup constant etc. The theoretical optimum suggested by the theory should therefore be foremost understood as maximal destabilization with remaining guaranteed robustness.
\answer{Practically optimal tuning should be expected to lie between $\beta=0$ and $\beta = \frac{1}{2}\frac{d\alpha^2}{2\mu + d\lambda}$. This behavior is observed in the considered numerical test cases in the following section.}
\end{remark}

\begin{remark}[Extension to non-trivial boundary conditions\label{remark:bc}]
The splitting strategy and its theoretical analysis in Lemma~\ref{lemma:splitting} are independent of any chosen boundary conditions. As the analysis is essentially considering the error equations~\eqref{eq:split-convergence-proof-1}, all right hand sides (also those that would be introduced as surface integrals weakly encoding boundary conditions) cancel and all functions involved are error terms canceling any strongly encoded boundary conditions, reducing the discussions -- as above -- to function spaces with zero traces. Overall, the convergence rate and deduction of a tuned (de-)stabilization parameter holds in general.
\end{remark}

\section{Numerical results}\label{sec:numerics}

\answer{We consider in total four numerical examples in both two and three dimensions to verify our theory.} \answer{In 2D, two examples} are designed to have an analytical solution allowing for assessing the spatial approximation properties. The first example is using Example 8.1 in~\cite{ahmed2019adaptive} and is based on a manufactured solution, while the second example is the classical Mandel benchmark problem, cf., e.g.,~\cite{mikelic2014numerical}. By varying material parameters, the problem has either a loosely or tightly coupling which enables a fair assessment of the performance of iterative splitting methods, its dependence on material parameters and the potential of improving it through \answer{tuned} destabilization. \answer{In 3D, we consider a similar manufactured solution to discuss approximation quality and a footing problem to discuss splitting methods, both in similar way as in 2D.}

For the spatial discretization, we use Raviart--Thomas spaces of order \(k\) for the Darcy velocity and for each row of the total stress, discontinuous piecewise polynomials of degree \(k\) for the scalar pressure, discontinuous vector-valued piecewise polynomials of degree \(k\) for the displacement, and continuous piecewise polynomials of degree \(k\) for the rotation, with \(k\geq 1\).

All computations were carried out in Python using the finite element library \texttt{NGSolve}/\texttt{Netgen} for the spatial discretization and assembly of the discrete systems~\cite{ngsolve}. Auxiliary tasks, including array handling, sparse matrix operations, and root finding for the analytical Mandel solution, were implemented using \texttt{NumPy} and \texttt{SciPy}~\cite{numpy2020,scipy2020}. \answer{All linear problems are solved using direct solvers, where in 2D SuperLU factorization is employed, and in 3D \texttt{PETSc}~\cite{petsc-web-page,petsc4py} is used with the multifrontal solver \texttt{MUMPS}~\cite{MUMPS2001,MUMPS2006} under a nested-dissection ordering.}

\subsection{Manufactured solution on the unit square}\label{subsec:manufactured}
We consider on the unit square $\Omega = [0,1]^2$ the manufactured polynomial solution
\begin{equation}
\mathbf{u}(\mathbf{x}, t) :=
\begin{pmatrix}
t x(1-x)y(1-y) \\
t y(1-y)x(1-x)
\end{pmatrix},
\qquad
p(\mathbf{x}, t) := t x(1-x)y(1-y),
\qquad 0 \le t \le 1.
\label{eq:manufactured-solution}
\end{equation}
The source terms, initial conditions, and Dirichlet boundary conditions are chosen such that \eqref{eq:manufactured-solution} is the exact solution of the continuous Biot system in strong form, cf.\ section~\ref{sec:porousmedia}.

To systematically investigate the effect of the coupling strength and assess the robustness with respect to key parameters (related to compressibility and permeability), we reduce the overall parameter dependence to two positive scaling parameters \answer{$\gamma_1$ and $\gamma_2$}. More precisely, we choose
\[
\answer{\kappa = \frac{\gamma_1}{\gamma_2}},
\qquad
c_0 = \gamma_1,
\qquad
\alpha = 1,
\qquad
\mu = 0.6,
\qquad
\lambda = 0.6\,\gamma_2.
\]
Here, $\gamma_1$ mainly scales the storage coefficient and the permeability, \answer{while 
$\gamma_2$ controls the compressibility and permeability of the solid skeleton}. Together they control the critical coupling strength between flow and mechanics, \answer{as defined in~\cite{kim2011stability} and occurring in the convergence rate in Lemma~\ref{lemma:splitting}},

\answer{\[\tau := \frac{d\alpha^2}{c_0\,(2\mu+d\lambda)}= \frac{\alpha^2}{c_0 K_\mathrm{dr}}\quad\text{for}\quad K_\mathrm{dr}:= \frac{2\mu}{d}+\lambda.\]}

Small values of $\tau < 1$ correspond to weak coupling, whereas large values of $\tau\geq 1$ correspond to strongly coupled regimes. In the experiments below, we vary
\[
\gamma_1 \in \{10^{-3},10^{-2},10^{-1},1,10\},
\qquad
\gamma_2 \in \{10^{-2},10^{-1},1,10,10^2,10^3,10^4\},
\]
and we consider polynomial orders $k=1,2,3$. This allows us to examine both the approximation properties of the method and the dependence of the iterative solver on the discretization order.

For the iterative solver, we compare two splitting schemes \answer{by varying the stabilization parameter $\beta$}, namely the classical fixed-stress split ($\beta=\beta_\mathrm{FS}:=0$) and the \answer{maximally destabilized} fixed-stress split \answer{($\beta = \beta_\mathrm{max} := \tfrac{1}{2}M$)} suggested by the theory, cf. Lemma~\ref{lemma:splitting}. At each time step, the splitting iteration is terminated once the relative change between two successive iterates is below a prescribed tolerance for both subproblems. More precisely, denoting by
\[
X_{\mathrm f}^m := (\velocity_h^m,p_h^m),
\qquad
X_{\mathrm m}^m := (\totalstress_h^m,\displacement_h^m,\vorticity_h^m)
\]
the flow and mechanics iterates at iteration $m$, respectively, we stop as soon as 

\begin{equation}
\frac{\|X_{\mathrm f}^{m}-X_{\mathrm f}^{m-1}\|_2}{\|X_{\mathrm f}^{m-1}\|_2+\varepsilon_0}
\le \texttt{tol}
\qquad\text{and}\qquad
\frac{\|X_{\mathrm m}^{m}-X_{\mathrm m}^{m-1}\|_2}{\|X_{\mathrm m}^{m-1}\|_2+\varepsilon_0}
\le \texttt{tol},
\label{eq:stopping-criterion}
\end{equation}
where \(\texttt{tol}=10^{-6}\) and \(\varepsilon_0=10^{-14}\) is a small safeguard parameter. Thus, \eqref{eq:stopping-criterion} measures the relative update between successive fixed-point iterates rather than the algebraic residual of the monolithic coupled system.

For all computations, the unit square is discretized by a structured triangulation obtained from an \(N\times N\) subdivision, so that the mesh size satisfies \(h\sim N^{-1}\). In the iteration-count experiments, we keep the spatial and temporal discretizations fixed and choose \(N=16\) and \(\Delta t=1/4\) on the time interval \([0,1]\). In the convergence study, we use the same time interval and time step, but successively refine the mesh and evaluate the spatial approximation errors at the final time \(t=1\).

We first examine the spatial approximation behavior for the reference choice $\gamma_1=\gamma_2=1$ without tuning, that is, with $\beta=\beta_\mathrm{FS}$. The convergence results are displayed in Figure~\ref{fig:analytical_convergencerates}. 

\begin{figure}[!ht]
    \centering    
    \begin{subfigure}[b]{0.49\textwidth}
        \centering
        \includegraphics[width=\textwidth]{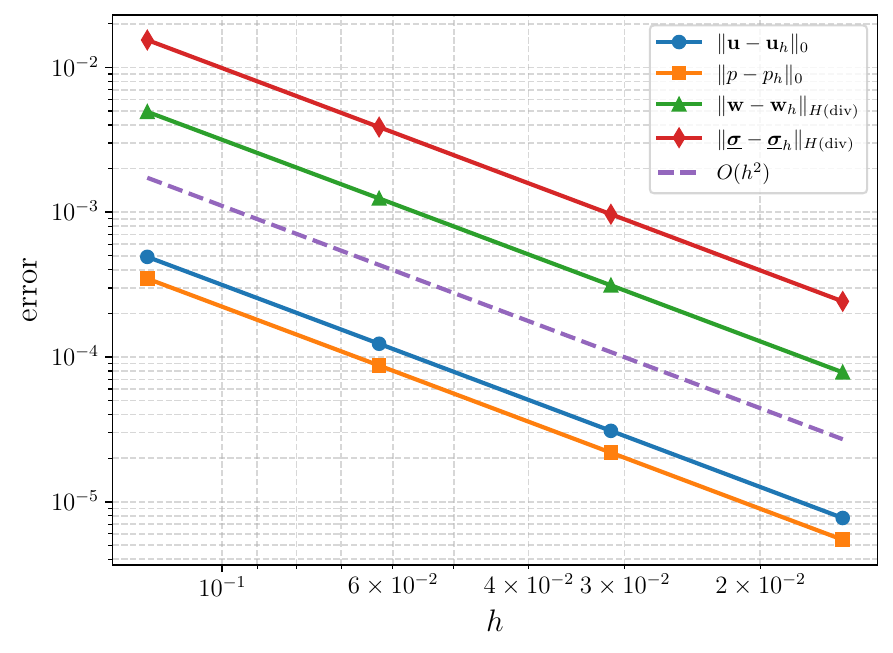} 
        \caption{$k=1$}
    \end{subfigure}
    \hfill    
    \begin{subfigure}[b]{0.49\textwidth}
        \centering
        \includegraphics[width=\textwidth]{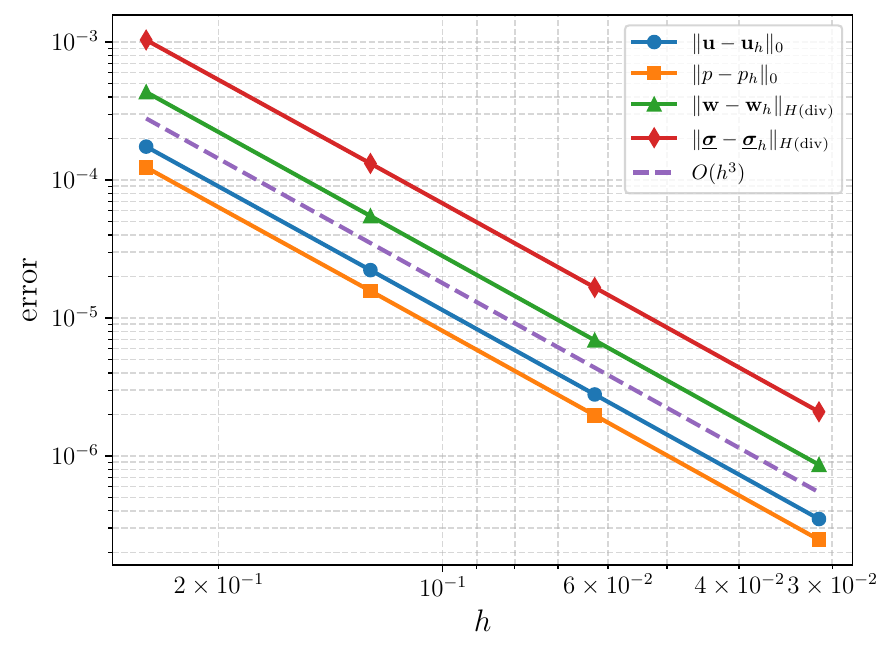}
        \caption{$k=2$}
    \end{subfigure}
    \vfill
        \begin{subfigure}[b]{0.49\textwidth}
        \centering
        \includegraphics[width=\textwidth]{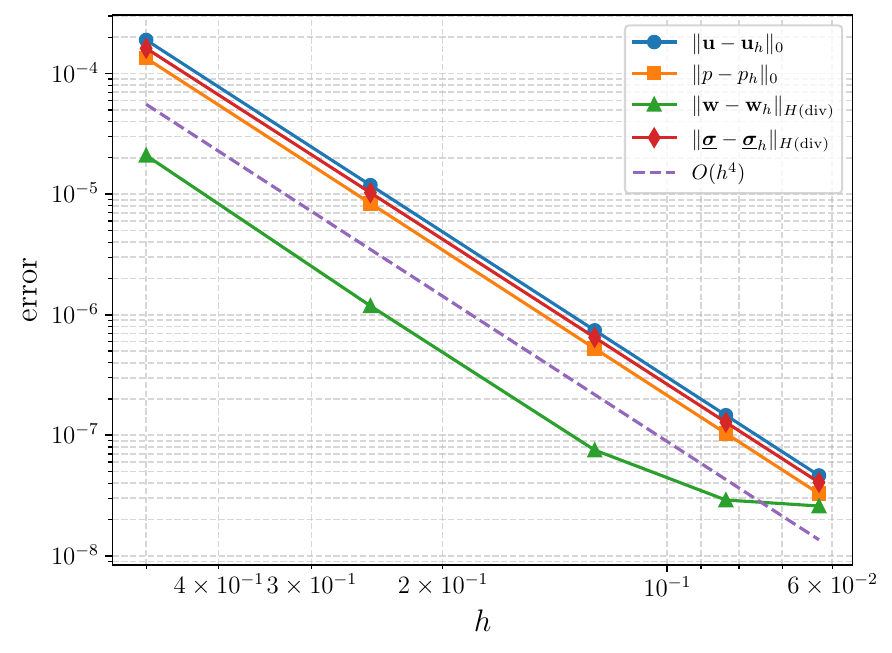} 
        \caption{$k=3$}
    \end{subfigure}
    \hfill    
    \begin{subfigure}[b]{0.49\textwidth}
        \centering
        \includegraphics[width=\textwidth]{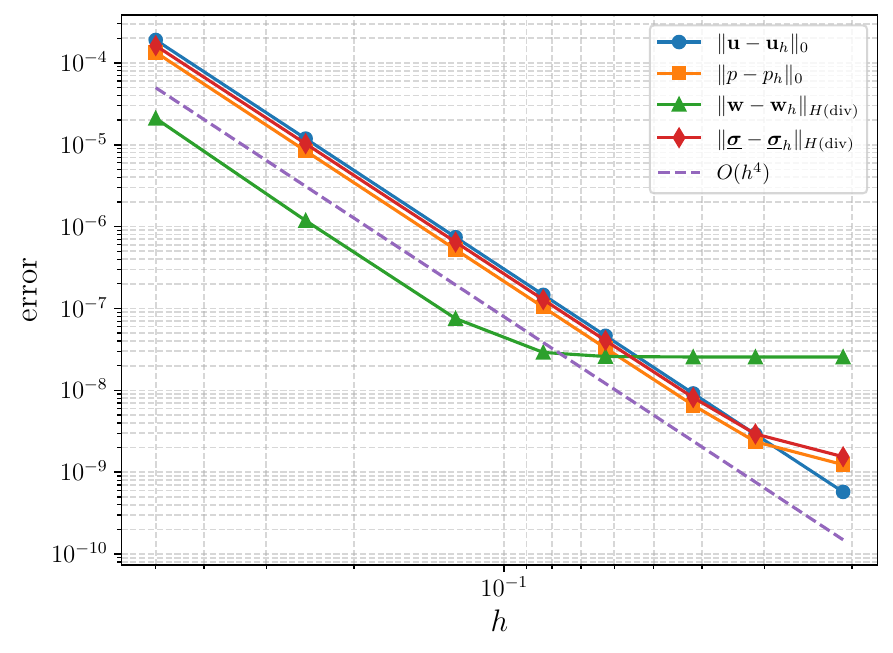}
        \caption{$k=3$, further mesh refinement}
    \end{subfigure}
   \caption{Spatial convergence for the manufactured solution with $\gamma_1=\gamma_2=1$ and $\beta=0$. The errors are evaluated at the final time $t=1$ for discrete space orders $k=1,2,3$.}
    \label{fig:analytical_convergencerates}
\end{figure}

\par
For $k=1$ and $k=2$, the observed decay is consistent with approximately $\mathcal{O}(h^{k+1})$ for all primary variables. For $k=3$, the same convergence behavior can be observed on the coarser mesh sequence. For finer meshes, some of the error curves no longer decrease with the expected rate. This is due to other error contributions, such as the fixed stopping tolerance in the splitting iteration and floating-point effects, and not to a loss of consistency of the discretization; ideally, tolerances for stopping criteria are correlated with discretization errors~\cite{jiranek2010posteriori}. Overall, the figure confirms the expected optimal convergence behavior of the mixed method for displacement, pressure, velocity, and total stress.

We next turn to the iterative behavior of the splitting scheme and report the average number of iterations per time step required to satisfy \eqref{eq:stopping-criterion}. Table~\ref{tab:avg_iter_gamma_scaling2_1} shows the results obtained by varying $\gamma_1$ while keeping $\gamma_2=1$. Since in this case decreasing $\gamma_1$ decreases both the storage coefficient $c_0$ and the permeability scale, the coupling strength $\tau$ increases substantially. The iteration counts reflect this trend clearly: the problem moves from a weakly coupled regime, where both schemes converge in only a few iterations, to a strongly coupled regime, where the plain splitting ($\beta=\beta_\mathrm{FS}$) becomes significantly more expensive. For this case, the benefit of tuning becomes more pronounced as the coupling becomes stronger. In particular, for $\gamma_1=10^{-3}$ the average number of iterations is reduced from about $180$ to about $96$ by the tuned scheme ($\beta=\beta_\mathrm{max}$). 

A complementary picture is obtained by varying $\gamma_2$ while keeping $\gamma_1=1$, see Table~\ref{tab:avg_iter_scaling2_gamma1}. Here, increasing $\gamma_2$ increases the Lam\'e parameter $\lambda$, and therefore drives the mechanics toward the nearly incompressible regime, with Poisson ratio approaching $1/2$. At the same time, since $\kappa=\gamma_1/\gamma_2$, large values of $\gamma_2$ also decrease the permeability. In terms of the coupling strength $\tau$, larger $\gamma_2$ correspond to weaker coupling, and the iteration counts indeed become very small in this regime. For smaller values of $\gamma_2$, the coupling strength increases toward and beyond $\tau\approx 1$, and the tuned scheme ($\beta=\beta_\mathrm{max}$) again yields a systematic reduction of the iteration counts.

At the same time, the large-$\gamma_2$ regime should be interpreted with some care. The manufactured solution \eqref{eq:manufactured-solution} remains an exact solution for all parameter choices by construction, so the test is mathematically fully valid throughout. However, for very large $\gamma_2$ it becomes less representative of the characteristic behavior of a genuinely nearly incompressible poroelastic response, since the spatial and temporal profiles of $\velocity$ and $p$ are fixed a priori while the material coefficients are varied over several orders of magnitude. In that sense, the example is best viewed as a controlled benchmark for discretization and iteration rather than as a physically realistic model in the extreme nearly incompressible limit.

This observation is also relevant when interpreting the velocity errors. In the large-$\gamma_2$ regime, the permeability is very small, so the Darcy velocity itself becomes small in magnitude, while its divergence must still balance the mass conservation constraint. In such a setting, accurate control of the $H(\div)$-error of the velocity can become more delicate, especially for a manufactured solution that is not specifically tailored to this singularly perturbed regime. Thus, a deterioration in the $H(\div)$-velocity error for very large $\gamma_2$ should not necessarily be interpreted as a structural deficiency of the method, but rather as a consequence of combining an extreme parameter regime with a fixed manufactured solution.

\begin{table}[ht!]
\renewcommand{\arraystretch}{1.15}
\centering
\begin{tabular}{cccccccc}
\toprule
$\gamma_1$ & coupling strength $\tau$ & \multicolumn{2}{c}{$k=1$} & \multicolumn{2}{c}{$k=2$} & \multicolumn{2}{c}{$k=3$} \\
& & $\beta_\mathrm{FS}$ & $\beta_\mathrm{max}$ & $\beta_\mathrm{FS}$ & $\beta_\mathrm{max}$ & $\beta_\mathrm{FS}$ & $\beta_\mathrm{max}$\\
\midrule
10 & $8.333 \times 10^{-2}$ & 4.25 & 4.00 & 4.25 & 4.00 & 4.25 & 4.00 \\
1 & $8.333 \times 10^{-1}$ & 7.25 & 6.00 & 7.25 & 6.00 & 7.25 & 6.00 \\
0.1 & $8.333 \times 10^{0}$ & 18.75 & 11.50 & 18.75 & 11.50 & 18.75 & 11.50 \\
0.01 & $8.333 \times 10^{1}$ & 60.50 & 33.00 & 60.50 & 33.00 & 60.50 & 33.00 \\
0.001 & $8.333 \times 10^{2}$ & 182.50 & 98.50 & 178.50 & 96.25 & 178.50 & 96.25 \\
\bottomrule
\end{tabular}
\caption{Average number of iterations per time step for varying $\gamma_1$ with fixed $\gamma_2=1$.}
\label{tab:avg_iter_gamma_scaling2_1}
\end{table}

\begin{table}[ht!]
\renewcommand{\arraystretch}{1.15}
\centering
\begin{tabular}{ccccccccc}
\toprule
$\gamma_2$ & Poisson ratio & coupling strength $\tau$ & \multicolumn{2}{c}{$k=1$} & \multicolumn{2}{c}{$k=2$} & \multicolumn{2}{c}{$k=3$} \\
& & & $\beta_\mathrm{FS}$ & $\beta_\mathrm{max}$ & $\beta_\mathrm{FS}$ & $\beta_\mathrm{max}$ & $\beta_\mathrm{FS}$ & $\beta_\mathrm{max}$\\
\midrule
10000 & 0.49995 & $1.667 \times 10^{-4}$ & 3.00 & 4.00 & 3.00 & 4.00 & 3.00 & 4.00 \\
1000 & 0.4995 & $1.665 \times 10^{-3}$ & 4.00 & 4.50 & 4.00 & 4.50 & 4.00 & 4.50 \\
100 & 0.49505 & $1.650 \times 10^{-2}$ & 4.50 & 5.25 & 4.50 & 5.25 & 4.50 & 5.25 \\
10 & 0.45455 & $1.515 \times 10^{-1}$ & 6.25 & 6.00 & 6.25 & 6.00 & 6.25 & 6.00 \\
1 & 0.25 & $8.333 \times 10^{-1}$ & 7.25 & 6.00 & 7.25 & 6.00 & 7.25 & 6.00 \\
0.1 & 0.04545 & $1.515 \times 10^{0}$ & 5.00 & 4.25 & 5.00 & 4.25 & 5.00 & 4.25 \\
0.01 & 0.00495 & $1.650 \times 10^{0}$ & 4.00 & 3.50 & 4.00 & 3.50 & 4.00 & 3.50 \\
\bottomrule
\end{tabular}
\caption{Average number of iterations per time step for varying $\gamma_2$ with fixed $\gamma_1=1$.}
\label{tab:avg_iter_scaling2_gamma1}
\end{table}

A further important observation in both Tables~\ref{tab:avg_iter_gamma_scaling2_1} and \ref{tab:avg_iter_scaling2_gamma1} is that the iteration counts are nearly identical for $k=1,2,3$. Hence, for this example, the convergence of the splitting iteration is governed primarily by the physical parameters and the coupling strength, whereas the polynomial degree mainly affects the spatial accuracy of the discrete solution. This separation is favorable in practice: one may increase the approximation order to improve accuracy without significantly affecting the convergence of the iterative coupling scheme.

In summary, the experiments confirm two main points. First, the proposed discretization exhibits the expected spatial convergence behavior for all primary variables. Second, a non-zero tuning parameter in the splitting scheme is particularly beneficial in strongly coupled regimes, where it leads to a substantial reduction in iteration counts, while in weakly coupled regimes both variants perform comparably. Moreover, no significant difference across discretizations is encountered in line with the convergence result in Lemma~\ref{lemma:splitting}, which is agnostic to the spatial discretization.

\subsection{Mandel's problem}\label{sec:mandel}
Mandel’s problem is a classical benchmark in poroelasticity for assessing how accurately a numerical method captures the coupling between elastic deformation and pore-pressure diffusion. It is particularly well known for the non-monotonic evolution of the pore pressure near the center of the specimen. In this study, we consider a rectangular fluid-saturated poroelastic sample of width $2a$ and height $2b$, subjected to an instantaneous compressive load through rigid plates attached to the top and bottom boundaries. The lateral boundaries are drained, and gravitational effects are neglected. Owing to the availability of exact analytical expressions for the pressure, displacement, Darcy velocity, and stress fields, this benchmark is widely used for the validation of coupled hydro-mechanical models.

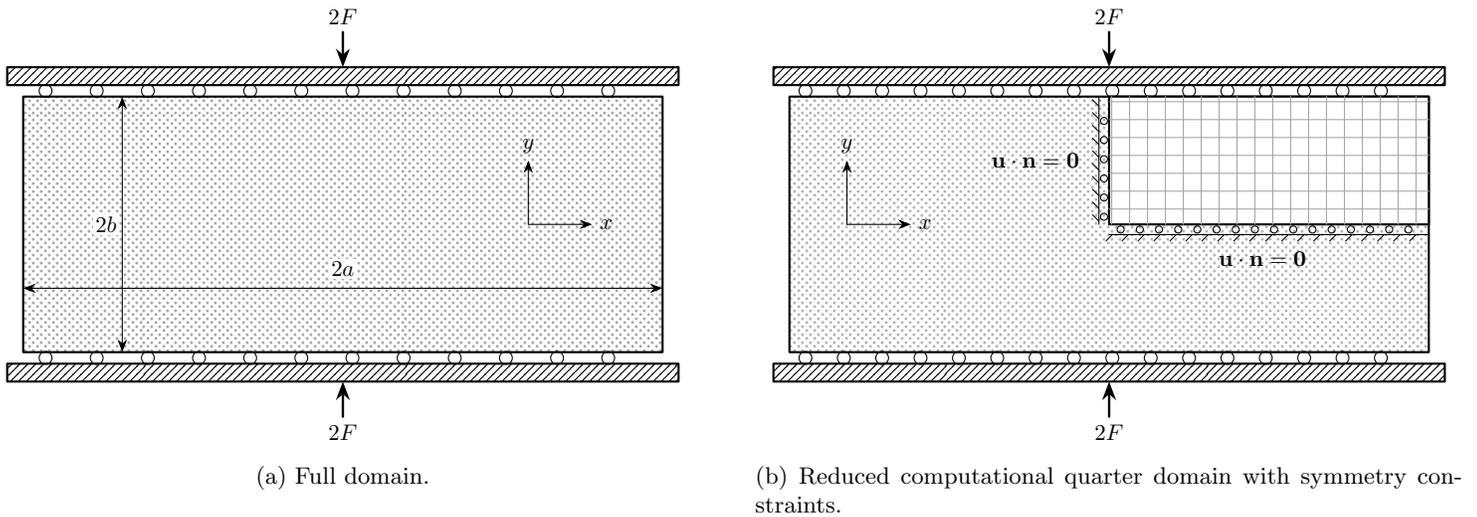
\begin{figure}[!ht]
\centering
\begin{subfigure}[t]{0.48\textwidth}
\centering
\begin{tikzpicture}[scale=0.85, transform shape, >=Stealth, line cap=round, line join=round]

\def\L{10.0}      
\def\H{4.0}       
\def\plategap{0.18}
\def\platethick{0.28}
\def\r{0.10}      

\fill[pattern=crosshatch dots, pattern color=gray!70] (0,0) rectangle (\L,\H);
\draw[thick] (0,0) rectangle (\L,\H);

\fill[pattern=north east lines] (-0.25,\H+\plategap) rectangle (\L+0.25,\H+\plategap+\platethick);
\draw[thick] (-0.25,\H+\plategap) rectangle (\L+0.25,\H+\plategap+\platethick);

\fill[pattern=north east lines] (-0.25,-\plategap-\platethick) rectangle (\L+0.25,-\plategap);
\draw[thick] (-0.25,-\plategap-\platethick) rectangle (\L+0.25,-\plategap);

\foreach \x in {0.35,1.15,...,9.65}{
    \draw (\x,\H+\plategap/2) circle (\r);
    \draw (\x,-\plategap/2) circle (\r);
}

\draw[->,line width=0.9pt] (\L/2,\H+\plategap+\platethick+0.55) -- (\L/2,\H + \platethick + \plategap);
\node[above] at (\L/2,\H+\plategap+\platethick+0.55) {$2F$};

\draw[->,line width=0.9pt] (\L/2,-\plategap-\platethick-0.55) -- (\L/2,-\plategap - \platethick);
\node[below] at (\L/2,-\plategap-\platethick-0.55) {$2F$};

\draw[<->,thin] (0,\H/4) -- (\L,\H/4);
\node[fill=white,inner sep=1pt] at (\L/2,\H/4 + 0.3) {$2a$};

\draw[<->,thin] (1.55,0) -- (1.55,\H);
\node[fill=white,inner sep=1pt,left] at (1.5,\H/2) {$2b$};

\draw[->] (7.9,2.0) -- (8.9,2.0);
\draw[->] (7.9,2.0) -- (7.9,3.0);
\node[right] at (8.9,2.0) {$x$};
\node[above] at (7.9,3.0) {$y$};

\end{tikzpicture}
\caption{Full domain.}
\label{fig:mandel_full_domain}
\end{subfigure}
\hfill
\begin{subfigure}[t]{0.48\textwidth}
\centering

\begin{tikzpicture}[scale=0.85, transform shape, >=Stealth, line cap=round, line join=round]

\def\L{10.0}         
\def\H{4.0}          
\def\plategap{0.18}
\def\platethick{0.28}
\def\r{0.10}

\def\xq{5.0}         
\def\yq{2.0}         

\fill[pattern=crosshatch dots, pattern color=gray!60] (0,0) rectangle (\L,\H);
\draw[thick] (0,0) rectangle (\L,\H);

\fill[pattern=north east lines] (-0.25,\H+\plategap) rectangle (\L+0.25,\H+\plategap+\platethick);
\draw[thick] (-0.25,\H+\plategap) rectangle (\L+0.25,\H+\plategap+\platethick);

\fill[pattern=north east lines] (-0.25,-\plategap-\platethick) rectangle (\L+0.25,-\plategap);
\draw[thick] (-0.25,-\plategap-\platethick) rectangle (\L+0.25,-\plategap);

\foreach \x in {0.25,0.85,...,9.85}{
    \draw (\x,\H+\plategap/2) circle (\r);
    \draw (\x,-\plategap/2) circle (\r);
}

\draw[->,line width=0.9pt] (\L/2,\H+\plategap+\platethick+0.55) -- (\L/2,\H+\plategap+\platethick);
\node[above] at (\L/2,\H+\plategap+\platethick+0.55) {$2F$};

\draw[->,line width=0.9pt] (\L/2,-\plategap-\platethick-0.55) -- (\L/2,-\plategap-\platethick);
\node[below] at (\L/2,-\plategap-\platethick-0.55) {$2F$};

\def\xq{5.0}
\def\yq{2.0}

\def\qgap{0.16}   
\def\qr{0.055}    

\fill[white] (\xq,\yq) rectangle (\L,\H);
\draw[thick] (\xq,\yq) rectangle (\L,\H);
\draw[step=0.28,thin,gray!70] (\xq,\yq) grid (\L,\H);

\draw[thin] (\xq-\qgap,\yq) -- (\xq-\qgap,\H);

\foreach \yy in {2.05,2.25,...,3.85}{
    \draw[thin] (\xq-\qgap,\yy) -- (\xq-\qgap-0.10,\yy+0.10);
}

\foreach \yy in {2.12,2.42,...,3.82}{
    \draw (\xq-\qgap/2,\yy) circle (\qr);
}

\draw[thin] (\xq,\yq-\qgap) -- (\L,\yq-\qgap);

\foreach \xx in {5.05,5.30,...,9.90}{
    \draw[thin] (\xx,\yq-\qgap) -- (\xx-0.10,\yq-\qgap-0.10);
}

\foreach \xx in {5.18,5.48,...,9.78}{
    \draw (\xx,\yq-\qgap/2) circle (\qr);
}

\node[left] at (\xq-\qgap-0.18,3.0) {$\displacement \cdot \bf n=0$};
\node[below] at (7.4,\yq-\qgap-0.14) {$\displacement \cdot \bf n=0$};

\draw[->] (0.9,2.0) -- (1.9,2.0);
\draw[->] (0.9,2.0) -- (0.9,3.0);
\node[right] at (1.9,2.0) {$x$};
\node[above] at (0.9,3.0) {$y$};
\end{tikzpicture}
\caption{Reduced computational quarter domain with symmetry constraints.}
\label{fig:mandel_reduced_domain}
\end{subfigure}

\caption{Schematic illustration of Mandel's problem and the reduced computational quarter domain used in the simulations.}
\label{fig:mandel_setup}
\end{figure}

\par 
We consider the classical Mandel benchmark problem; see, for example, \cite{mikelic2014numerical}. Owing to symmetry, the problem is solved on the quarter domain $\Omega=(0,a)\times(0,b)$. In our computations, we prescribe the following initial conditions:

\begin{equation}\label{eq:mandel_ic}
\begin{aligned}
p(x,0)
&=
\frac{FB(1+\nu_u)}{3a},
&& x\in (0,a),\\[0.5ex]
\totalstress_{11}(x,0) &= 
\totalstress_{12}(x,0) = \totalstress_{21}(x,0) = 0,
&& x\in (0,a),\\[0.5ex]
\totalstress_{22}(x,0)
&=
-\frac{F}{a},
&& x\in (0,a),
\end{aligned}
\end{equation}
and the boundary conditions 
\begin{equation}\label{eq:mandelbc}
\begin{alignedat}{4}
p &= 0, 
&\qquad & 
&\qquad \answer{\totalstress \bf n} &= 0, 
&\qquad &\text{on } x=a, \\
\velocity \cdot \bf n &= 0, 
&\qquad \displacement \cdot \bf n &= \displacement_{2}^{\mathrm{analytical}}(b,t), 
&\qquad \totalstress_{12} &= 0, 
&\qquad &\text{on } y=b, \\
\velocity \cdot \bf n &= 0, 
&\qquad \displacement \cdot \bf n &= 0, 
&\qquad \totalstress_{21} &= 0, 
&\qquad &\text{on } x=0, \\
\velocity \cdot \bf n &= 0, 
&\qquad \displacement \cdot \bf n &= 0, 
&\qquad \totalstress_{12} &= 0, 
&\qquad &\text{on } y=0.
\end{alignedat}
\end{equation}

\par
This benchmark admits an exact series solution satisfying the above conditions; see \cite{abousleiman1996mandel}. In the notation adopted in this paper, the exact pressure and displacement are given by
\begin{align}
p(x,t)
&=
\frac{2FB(1+\nu_u)}{3a}
\sum_{n=1}^{\infty}
\frac{\sin(\alpha_n)}
{\alpha_n-\sin(\alpha_n)\cos(\alpha_n)}
\left(
\cos\!\left(\frac{\alpha_n x}{a}\right)-\cos(\alpha_n)
\right)
e^{-\frac{\alpha_n^2 c t}{a^2}}, \nonumber
\\[1ex]
\displacement_1(x,t)
&=
\left[
\frac{F\nu}{2\mu a}
-
\frac{F\nu_u}{\mu a}
\sum_{n=1}^{\infty}
\frac{\sin(\alpha_n)\cos(\alpha_n)}
{\alpha_n-\sin(\alpha_n)\cos(\alpha_n)}
e^{-\frac{\alpha_n^2 c t}{a^2}}
\right] x
+
\frac{F}{\mu}
\sum_{n=1}^{\infty}
\frac{\cos(\alpha_n)}
{\alpha_n-\sin(\alpha_n)\cos(\alpha_n)}
\sin\!\left(\frac{\alpha_n x}{a}\right)
e^{-\frac{\alpha_n^2 c t}{a^2}}, \nonumber
\\[1ex]
\displacement_2(y,t)
&=
\left[
-\frac{F(1-\nu)}{2\mu a}
+
\frac{F(1-\nu_u)}{\mu a}
\sum_{n=1}^{\infty}
\frac{\sin(\alpha_n)\cos(\alpha_n)}
{\alpha_n-\sin(\alpha_n)\cos(\alpha_n)}
e^{-\frac{\alpha_n^2 c t}{a^2}}
\right] y, \nonumber
\end{align}
while the analytical solutions for the total stress and Darcy velocity read
\begin{align}
\totalstress_{11} &= \totalstress_{12} = \totalstress_{21} = 0, \nonumber
\\[1ex]
\totalstress_{22}(x,t)
&=
-\frac{F}{a}
-\frac{2F}{a}
\sum_{n=1}^{\infty}
\frac{\sin(\alpha_n)}
{\alpha_n-\sin(\alpha_n)\cos(\alpha_n)}
\left[
\frac{\nu_u-\nu}{1-\nu}\,
\cos\!\left(\frac{\alpha_n x}{a}\right)
-\cos(\alpha_n)
\right]
e^{-\frac{\alpha_n^2 c t}{a^2}}, \nonumber
\\[1ex]
\velocity_1(x,t)
&=
\frac{2FB(1+\nu_u)k}{3a^2\eta}
\sum_{n=1}^{\infty}
\frac{\alpha_n\sin(\alpha_n)}
{\alpha_n-\sin(\alpha_n)\cos(\alpha_n)}
\sin\!\left(\frac{\alpha_n x}{a}\right)
e^{-\frac{\alpha_n^2 c t}{a^2}}, \nonumber
\\
\velocity_2(x,t) &= 0. \nonumber
\end{align}
Here, $\alpha_n$ denotes the positive roots of the transcendental equation
\begin{equation}
\tan(\alpha_n)
=
\frac{1-\nu}{\nu_u-\nu}\,\alpha_n,
\qquad n\in\mathbb{N}.
\label{eq:mandel_roots}
\end{equation}
The quantities $\nu$, $F$, $B$, $c$, and $a$ denote the benchmark parameters appearing in the closed-form solution. In our computations, their values are determined from the physical Biot parameters listed in Table~\ref{tab:mandel_parameters}.

\begin{table}[!ht] 
\renewcommand{\arraystretch}{1.15}
\centering
\caption{Input parameters for Mandel's problem}
\label{tab:mandel_parameters}
\begin{tabular}{lll}
\toprule
Symbol & Quantity & Value \\
\midrule
$a$ & Dimension in $x$ & $100\ \mathrm{m}$ \\
$b$ & Dimension in $y$ & $10\ \mathrm{m}$ \\
$E$ & Young's modulus & $5.94 \times 10^{9}\ \mathrm{Pa}$ \\
$\nu$ & Poisson's ratio & $0.2$ \\
$c_f$ & Fluid compressibility & $3.03 \times 10^{-10}\ /\mathrm{Pa}$ \\
$F$ & Applied load & $6.0 \times 10^{8}\ \mathrm{N\,m^{-1}}$ \\
$\alpha$ & Biot's constant & $1.0$ \\
$k$ & Permeability & $100\ \mathrm{mD}$ \\
$\phi$ & Initial porosity & $0.2$ \\
$\eta$ & Fluid viscosity & $1.0\ \mathrm{cP}$ \\
$\Delta x$ & Grid spacing in $x$ & $2.5\ \mathrm{m}$ \\
$\Delta y$ & Grid spacing in $y$ & $0   .25\ \mathrm{m}$ \\
$\Delta t$ & Time step size & $10\ \mathrm{s}$ \\
$t_T$ & Total simulation time & $50{,}000\ \mathrm{s}$ \\
$B$ & Skempton coefficient & $0.83333$ \\
$\nu_u$ & Undrained Poisson's ratio & $0.44$ \\
\answer{$M_{\mathrm b}$} & Biot's modulus & $1.65 \times 10^{10}\ \mathrm{Pa}$ \\
$c$ & Diffusivity coefficient & $0.465\ \mathrm{m}^2\ \mathrm{s}^{-1}$ \\
\bottomrule
\end{tabular}
\end{table}

\begin{figure}[!ht]
    \centering    
    \begin{subfigure}[b]{0.32\textwidth}
        \centering
        \includegraphics[width=\textwidth]{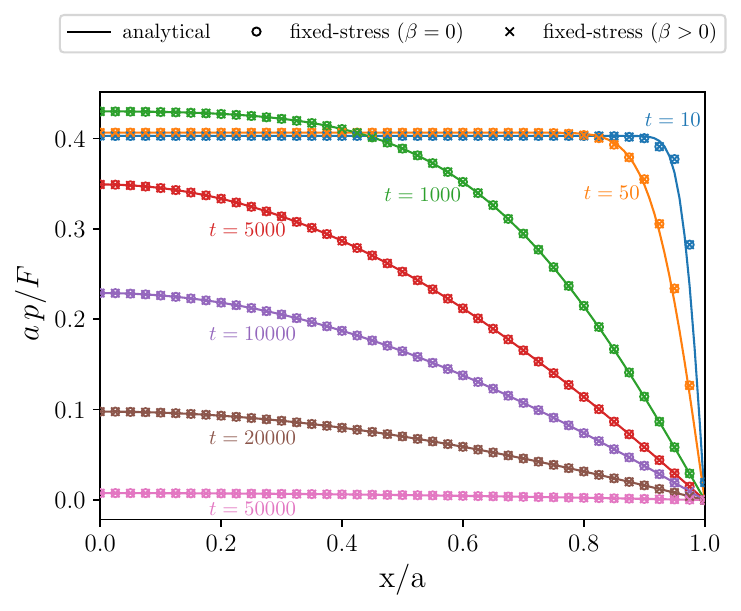} 
        \caption{Dimensionless pressure}
    \end{subfigure}
    \hfill    
    \begin{subfigure}[b]{0.32\textwidth}
        \centering
        \includegraphics[width=\textwidth]{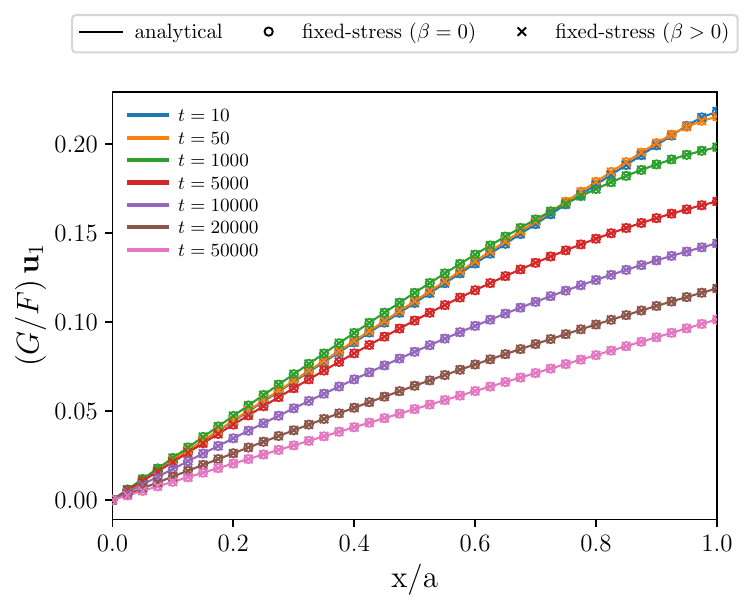}
        \caption{Dimensionless horizontal displacement}
    \end{subfigure}
    \hfill
        \begin{subfigure}[b]{0.32\textwidth}
        \centering
        \includegraphics[width=\textwidth]{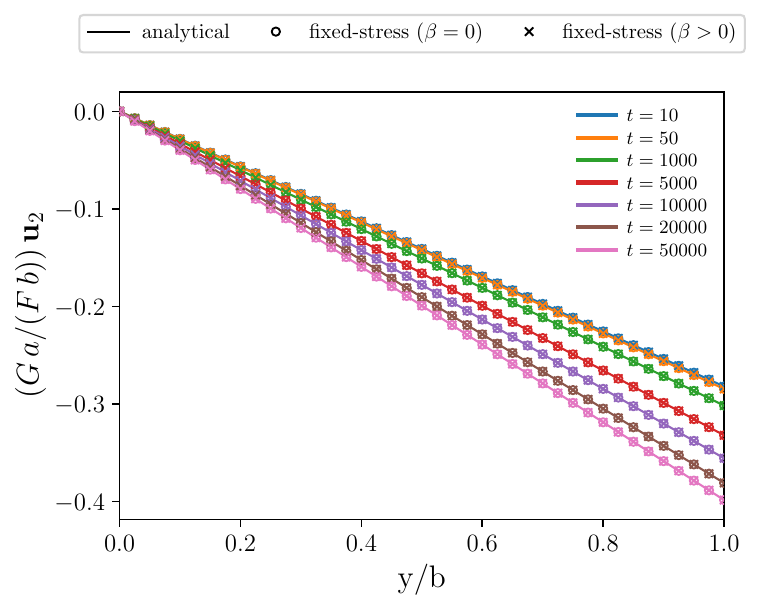} 
        \caption{Dimensionless vertical displacement}
    \end{subfigure}
    \vfill    
    \begin{subfigure}[b]{0.32\textwidth}
        \centering
        \includegraphics[width=\textwidth]{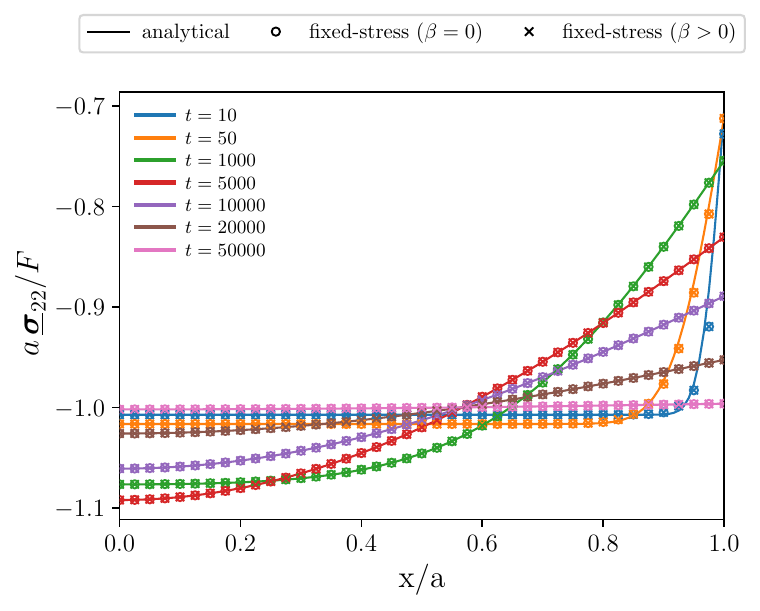}
        \caption{Dimensionless vertical total stress}
    \end{subfigure}
        \begin{subfigure}[b]{0.32\textwidth}
        \centering
        \includegraphics[width=\textwidth]{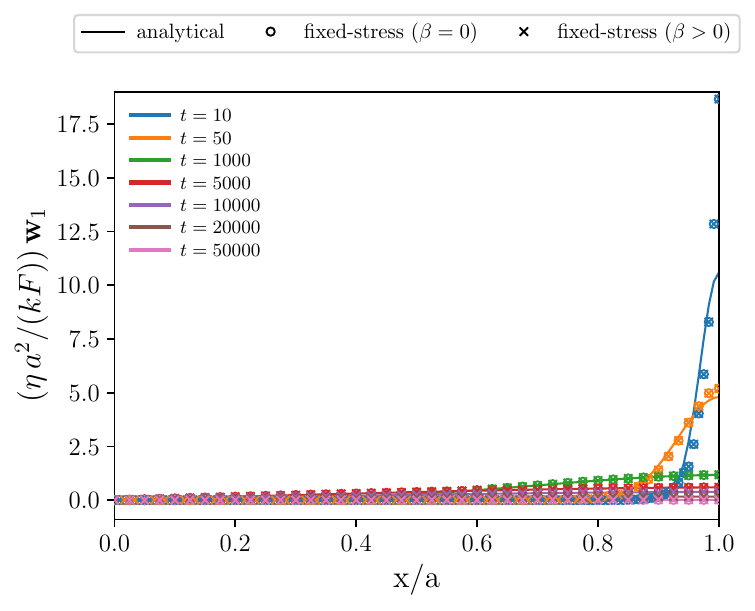} 
        \caption{Dimensionless horizontal velocity}
    \end{subfigure}
\caption{Dimensionless solution variables for Mandel's problem.}    \label{fig:mandelsolvar}
    
\end{figure}

For the numerical simulation of Mandel's problem shown in Figure~\ref{fig:mandelsolvar}, we use the parameter set listed in Table~\ref{tab:mandel_parameters}, the lowest discrete space order choice \(k=1\), a structured anisotropic mesh with \(\Delta x=2.5\) and \(\Delta y=0.25\), and a uniform time step \(\Delta t=10\) on the interval \([0,50000]\). \answer{For a more nuanced discussion in light of Remark~\ref{remark:optimality} on practical optimality, we resolve the admissible destabilization range by sweeping the fraction $\theta:=\beta/\beta_\mathrm{max}\in\left\{0,\tfrac14,\tfrac12,\tfrac34,1\right\}$. With this,  $\theta=0$ is the plain fixed-stress split and $\theta=1$ the choice $\beta=\beta_\mathrm{max}$ of Lemma~\ref{lemma:splitting}, which is the maximal destabilization with guaranteed robustness.}

\answer{This benchmark supports the statements made in Remark~\ref{remark:optimality}: the theoretically optimized destabilization need not be practically optimal. For the reference run of Figure~\ref{fig:mandelsolvar}, averaged over the full time interval $[0,50000]$, the theoretical optimum $\theta=1$ barely changes the count relative to the plain split, requiring $6.0072$ against $6.0402$ iterations per time step on average, whereas the intermediate choice $\theta=\tfrac34$ halves it to $3.0$. Thus theoretical tuning does not necessarily lead to improved performance in practice, while a moderate destabilization does. This should be interpreted in light of the fact that Mandel's problem, despite being posed in two dimensions, exhibits a strongly one-dimensional character, and the optimal tuning depends on factors not captured by the theory. The systematic dependence on $\theta$ across coupling strengths and polynomial orders is quantified in Table~\ref{tab:mandel_avg_iter_gamma3} below, over a shorter interval concentrated on the coupled transient.}

Figure~\ref{fig:mandelsolvar} shows the numerical and analytical solutions for the dimensionless pressure, displacements, vertical stress, and horizontal Darcy velocity at several time instances. Overall, the agreement is very good for all quantities. \answer{In particular, the tuning affects the iteration count and not the converged solution, so the profiles are indistinguishable across the sweep and reproducing the analytical solution indicating that the spatial discretization error, discussed further below, dominates the remaining algebraic error. The figure therefore displays only the two extremes, the plain split $\beta=0$ and the theoretical optimum $\beta=\beta_\mathrm{max}$, the latter labeled $\beta>0$ in the legend.}

The pressure plot in Figure~\ref{fig:mandelsolvar}(a) shows the characteristic behavior of Mandel's problem. At early times the pressure is largest in the interior and vanishes at the drained boundary $x=a$. As time evolves, the pressure near the center remains elevated for some time before eventually decaying, reflecting the classical Mandel effect. The displacement plots in Figures~\ref{fig:mandelsolvar}(b)--(c) are also consistent with the expected mechanics: the horizontal displacement increases with $x$, while the vertical displacement is compressive and varies essentially linearly in the vertical direction. Likewise, the vertical total stress in Figure~\ref{fig:mandelsolvar}(d) starts from the nearly uniform initial state and then develops the expected spatial variation induced by the coupled drainage process.

The largest visible discrepancy between analytical and discrete solutions occurs in the horizontal Darcy velocity near the drained boundary $x=a$ at early times; see Figure~\ref{fig:mandelsolvar}(e). This is not unexpected. The discrete velocity is approximated in an $H(\div)$-conforming Raviart--Thomas space, for which the natural control is in terms of fluxes, moments, and divergence, rather than pointwise values. Therefore, near a boundary layer or steep spatial variation, pointwise plots of the velocity may show an overshoot even when the integral behavior of the discrete flux is well captured.

In summary, this benchmark confirms that the proposed method reproduces the analytical Mandel solution accurately for all primary variables. \answer{For the parameter range considered here, the fixed-stress splitting converges reliably for every $\theta$ tested. While the theoretical optimum $\theta=1$ has only a minor influence on the iteration count compared with the untuned scheme, intermediate destabilizations reduce it substantially, as the scan below shows.}

As a final numerical experiment, we investigate how the performance of the splitting scheme changes when the Young's modulus and the applied load are scaled simultaneously. More precisely, we set
\[
E=\gamma_3 E_{\mathrm{ref}},
\qquad
F=\gamma_3 F_{\mathrm{ref}},
\]
with
\[
E_{\mathrm{ref}}=5.94\times 10^9,
\qquad
F_{\mathrm{ref}}=6.0\times 10^8,
\]

\answer{while updating all model parameters depending on \(E\) accordingly and keeping the remaining material parameters of Table~\ref{tab:mandel_parameters} fixed. Since the iteration count of this benchmark is largest during the coupled transient and remains low once the pressure has diffused, we restrict the present scan to the time interval $[0,100]$, so that the reported averages characterize the coupled regime rather than the near-steady tail; the averages of the reference run above, taken over the full time interval $[0,50000]$, are correspondingly lower.} This scaling keeps the solution variables on a broadly comparable scale across the different tests, while the coupling strength changes. The corresponding average numbers of iterations per time step are reported in Table~\ref{tab:mandel_avg_iter_gamma3}.

\answer{\begin{table}[ht!]
\renewcommand{\arraystretch}{1.15}
\centering
\begin{tabular}{ccccccccc}
\toprule
$\gamma_3$ & coupling strength $\tau$ & $k$ & \multicolumn{5}{c}{$\theta=\beta/\beta_\mathrm{max}$} \\
\cmidrule(lr){4-8}
 &  &  & $0$ & $0.25$ & $0.5$ & $0.75$ & $1$ \\
\midrule
\multirow{3}{*}{$10^{2}$} & \multirow{3}{*}{$4.000 \times 10^{-2}$}
   & 1 & 4.0 & 3.8 & 3.2 & \textbf{3.0} & 3.5 \\
 & & 2 & 4.0 & 3.8 & 3.2 & \textbf{3.0} & 3.5 \\
 & & 3 & 4.0 & 3.8 & 3.2 & \textbf{3.0} & 3.5 \\
\midrule
\multirow{3}{*}{$10^{1}$} & \multirow{3}{*}{$4.000 \times 10^{-1}$}
   & 1 & 5.8 & 5.1 & 4.7 & \textbf{3.0} & 5.0 \\
 & & 2 & 5.8 & 5.1 & 4.7 & \textbf{3.0} & 5.0 \\
 & & 3 & 5.8 & 5.1 & 4.7 & \textbf{3.0} & 5.0 \\
\midrule
\multirow{3}{*}{$10^{0}$} & \multirow{3}{*}{$4.000 \times 10^{0}$}
   & 1 & 9.0 & 7.4 & 6.1 & \textbf{3.0} & 7.1 \\
 & & 2 & 9.0 & 7.4 & 6.1 & \textbf{3.0} & 7.1 \\
 & & 3 & 9.0 & 7.4 & 6.1 & \textbf{3.0} & 7.1 \\
\midrule
\multirow{3}{*}{$10^{-1}$} & \multirow{3}{*}{$4.000 \times 10^{1}$}
   & 1 & 9.6 & 8.1 & 6.1 & \textbf{4.4} & 8.1 \\
 & & 2 & 9.6 & 8.1 & 6.1 & \textbf{3.3} & 8.0 \\
 & & 3 & 9.6 & 8.1 & 6.1 & \textbf{3.1} & 8.0 \\
\midrule
\multirow{3}{*}{$10^{-2}$} & \multirow{3}{*}{$4.000 \times 10^{2}$}
   & 1 & 9.2 & 7.5 & \textbf{6.2} & 8.1 & 18.8 \\
 & & 2 & 9.3 & 8.0 & \textbf{6.1} & 6.5 & 11.9 \\
 & & 3 & 9.4 & 8.0 & 6.1 & \textbf{5.0} & 7.9 \\
\midrule
\multirow{3}{*}{$10^{-3}$} & \multirow{3}{*}{$4.000 \times 10^{3}$}
   & 1 & 9.1 & \textbf{7.6} & 8.1 & 14.3 & 121.1 \\
 & & 2 & 9.2 & 7.6 & \textbf{7.4} & 11.9 & 54.9 \\
 & & 3 & 9.1 & 7.4 & \textbf{7.2} & 11.1 & 32.1 \\
\bottomrule
\end{tabular}
\caption{Mandel's problem: average number of iterations per time step on $[0,100]$ for varying $\gamma_3=E/E_{\mathrm{ref}}=F/F_{\mathrm{ref}}$ versus the destabilization fraction $\theta=\beta/\beta_\mathrm{max}$. Here $\theta=0$ is the plain split and $\theta=1$ the theoretical optimum $\beta_\mathrm{max}$; the per-row minimum is in bold.}
\label{tab:mandel_avg_iter_gamma3}
\end{table}}

\answer{Two statements hold across the entire parameter scan. First, every destabilization with $\theta\le\tfrac12$ improves on the plain split at every coupling strength and every order: $\theta=\tfrac14$ reduces the count in all eighteen rows, and $\theta=\tfrac12$ reduces it further in all but the most strongly coupled row at $k=1$, staying between $3.2$ and $8.1$ where the plain split requires $4.0$ to $9.6$. Second, the counts are unimodal in $\theta$, so that each row possesses a single practical minimum, and the ascent beyond it is far steeper than the descent towards it.}

\answer{The location of that minimum, however, moves with the coupling strength. For $\gamma_3\ge10^{-1}$, that is for the weakly and moderately coupled regimes, it sits at $\theta=\tfrac34$ and the gain is pronounced: $3.0$ iterations against $9.0$ for the plain split at $\gamma_3=1$. Beyond that it drifts downwards, to $\theta=\tfrac12$ at $\gamma_3=10^{-2}$ for $k=1,2$ and to $\theta=\tfrac14$ at $\gamma_3=10^{-3}$ for $k=1$, while for $k=3$ the $\theta=\tfrac34$ optimum survives one decade longer. The elevated optimum in the moderately coupled regime reflects the strongly one-dimensional stress character of Mandel's problem, in contrast to the genuinely three-dimensional footing configuration, where the minimum remains pinned at $\theta=\tfrac12$ across three decades of coupling.}

\answer{The theoretical optimum $\theta=1$ is again the least favorable of the destabilized choices once the coupling is strong. It is mildly beneficial for $\gamma_3\ge10^{-1}$, but at $\gamma_3=10^{-2}$ it already exceeds the plain split for $k=1,2$, and at $\gamma_3=10^{-3}$ it reaches $121.1$ iterations against $9.1$ for the plain split at $k=1$. This deterioration is relieved but not removed by increasing the order ($121.1$, $54.9$ and $32.1$ for $k=1,2,3$), consistent with the discrete stability constants entering the practical optimum but not the contraction factor of Lemma~\ref{lemma:splitting}. As in Section~\ref{subsec:footing_3d}, this is a concrete instance of Remark~\ref{remark:optimality}: $\beta_\mathrm{max}$ is the a-priori-robust maximal destabilization, sitting at the edge of the useful range rather than at its center.}

Thus, in contrast to the previous test with the manufactured solution, the present scaling reveals a more sensitive dependence of the iteration on both the coupling strength and the discretization order. 

\answer{In particular, once the coupling becomes sufficiently strong, the choice $\theta=1$ is no longer uniformly robust, while still converging, across all polynomial orders. The higher-order discretizations appear to handle this regime better, whereas for lower orders the same tuning may even deteriorate the convergence. This suggests that the discretization scheme impacts overall stability not captured by the theory.}

A further notable feature of Table~\ref{tab:mandel_avg_iter_gamma3} is that the untuned scheme remains relatively stable across the full range of $\gamma_3$, with iteration counts staying close to $9$ in the strongly coupled cases. By contrast, 

\answer{the maximal destabilized choice $\theta=1$} shows a much stronger dependence on both $\gamma_3$ and $k$. This indicates that, for this scaled Mandel test, the practical effect of tuning is more delicate than in the manufactured example, and further method characteristics are impacting the convergence rate which have not been picked up in Lemma~\ref{lemma:splitting}.

\answer{Overall, Table~\ref{tab:mandel_avg_iter_gamma3} suggests that negative tuning can improve convergence throughout: a moderate destabilization $\theta\le\tfrac12$ is beneficial in every regime tested.} In particular, for very small $\gamma_3$ the interaction between the parameter scaling, the induced coupling strength, and the spatial discretization becomes more subtle. This experiment with additional emphasis on physically motivated destabilization therefore complements the previous tests by showing that the effect of tuning can depend significantly on the underlying physical regime, consistent with previous analyses of the fixed-stress split for saddle-point formulations of the Biot equations~\cite{both2017numerical}. In addition, as discussed in~\cite{storvik2019optimization,storvik2020fixed}, stability properties of the discretization as the inf-sup constant may impact the performance of the splitting and the optimal tuning parameter.

\answer{\subsection{Manufactured solution on the unit cube}\label{subsec:manufactured_3d}}
We extend the manufactured-solution study of Section~\ref{subsec:manufactured} to the unit cube $\Omega=[0,1]^3$. With the bubble $b(\mathbf x)=x(1-x)\,y(1-y)\,z(1-z)$, which vanishes on $\partial\Omega$, we prescribe
\begin{equation}
\displacement(\mathbf x,t):=t\,\bigl(1,\,1,\,1\bigr)^{\!\top} b(\mathbf x),
\qquad
p(\mathbf x,t):=t\,b(\mathbf x),
\qquad 0\le t\le 1.
\label{eq:manufactured-solution-3d}
\end{equation}
The sources, the homogeneous Dirichlet data for $\displacement$ and $p$, and the initial condition are chosen so that \eqref{eq:manufactured-solution-3d} solves the continuous Biot system in strong form, cf.\ Section~\ref{sec:porousmedia}.

The discretization is the direct analogue of the two-dimensional one, now with three $H(\div)$ stress rows and a three-component discontinuous displacement. The essential difference lies in the rotation: in $d=3$ the weak symmetry constraint carries the three skew components $\operatorname{as}(\totalstress)=(\totalstress_{23}-\totalstress_{32},\ \totalstress_{31}-\totalstress_{13},\ \totalstress_{12}-\totalstress_{21})$, so that the rotation $\vorticity$ is a three-vector ($2d-3=3$), discretized in a continuous vector-valued $\mathbb{P}_k$ space. We report $k=1,2$.

The unit cube is discretized by a structured tetrahedral $N\times N\times N$ mesh ($h\sim N^{-1}$). For the study of the approximation quality of the fully-mixed finite element discretization, we take $\gamma_1=\gamma_2=1$, $\beta=0$, and $\Delta t=1/4$ on $[0,1]$, evaluating the errors at $t=1$; the results are shown in Figure~\ref{fig:analytical_convergencerates_3d}.

\begin{figure}[!ht]
    \centering
    \begin{subfigure}[b]{0.49\textwidth}
        \centering
        \includegraphics[width=\textwidth]{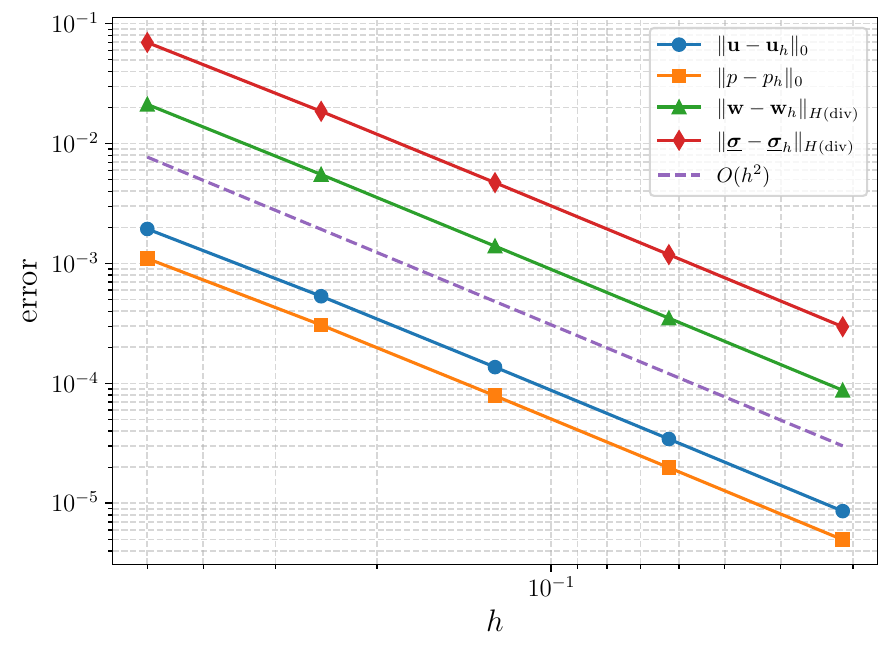}
        \caption{$k=1$}
    \end{subfigure}
    \hfill
    \begin{subfigure}[b]{0.49\textwidth}
        \centering
        \includegraphics[width=\textwidth]{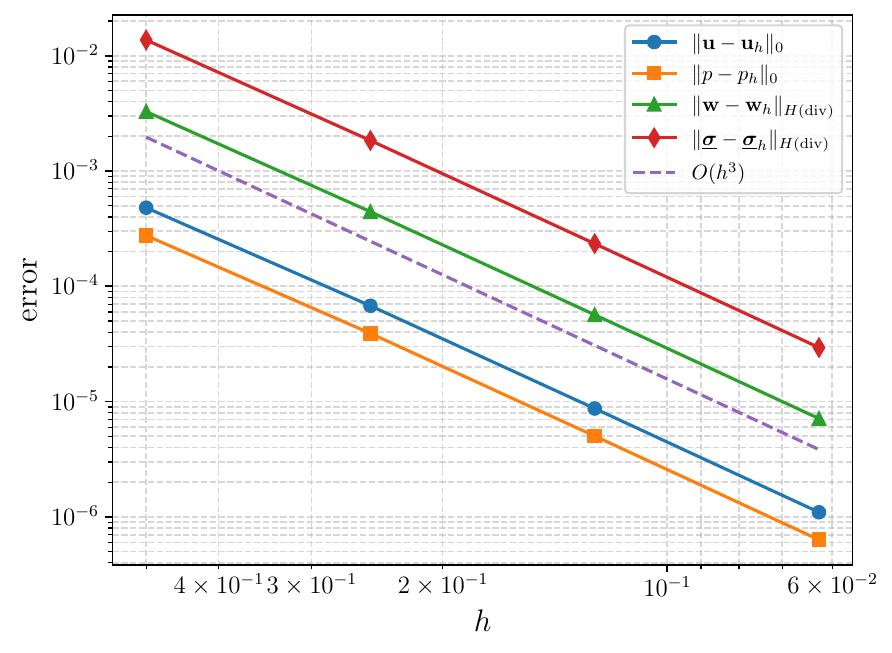}
        \caption{$k=2$}
    \end{subfigure}
    \caption{Spatial convergence for the three-dimensional manufactured solution \eqref{eq:manufactured-solution-3d} on the unit cube with $\gamma_1=\gamma_2=1$ and $\beta=0$. The errors are evaluated at the final time $t=1$ for discrete space orders $k=1,2$.}
    \label{fig:analytical_convergencerates_3d}
\end{figure}

For both $k=1$ and $k=2$ the errors of all primary variables decay at the optimal rate $\mathcal{O}(h^{k+1})$, namely displacement and pressure in $L^2(\Omega)$ and velocity and total stress in $H(\div,\Omega)$, with the total stress reaching order $k+1$ in $H(\div)$ since $\div\mathrm{RT}_k=\mathbb{P}_k$. The curves stay straight over the whole tested range, with no tolerance-induced flattening, confirming the optimal approximation behavior already observed in two dimensions.

We next repeat the two parameter scans of Section~\ref{subsec:manufactured}, using the same two-parameter family, stopping criterion~\eqref{eq:stopping-criterion}, and time step, and comparing the untuned ($\beta=\beta_\mathrm{FS}$) and tuned ($\beta=\beta_\mathrm{max}$) splits for $k=1,2$. Since the counts are mesh-independent (Lemma~\ref{lemma:splitting}), a fixed coarse mesh suffices. The coupling strength is $\tau=1$ at the reference choice $\gamma_1=\gamma_2=1$.

Table~\ref{tab:manufactured_gamma_scan_3d} varies $\gamma_1$ at $\gamma_2=1$, so that $\tau=1/\gamma_1$ ranges over $[10^{-1},10^{3}]$. The untuned count grows from about $4$ to about $168$ (for $k=1$), and tuning helps increasingly with the coupling: at $\gamma_1=10^{-3}$ it reduces the count by about $46\%$ ($168.50\to91.25$ for $k=1$, $153.50\to82.75$ for $k=2$), while at weak coupling the two variants coincide. As in two dimensions the counts barely depend on $k$.

\begin{table}[ht!]
\renewcommand{\arraystretch}{1.15}
\centering
\begin{tabular}{cccccc}
\toprule
$\gamma_1$ & coupling strength $\tau$ & \multicolumn{2}{c}{$k=1$} & \multicolumn{2}{c}{$k=2$} \\
 &  & $\beta_\mathrm{FS}$ & $\beta_\mathrm{max}$ & $\beta_\mathrm{FS}$ & $\beta_\mathrm{max}$ \\
\midrule
10 & $1.000 \times 10^{-1}$ & 4.25 & 4.00 & 4.25 & 4.00 \\
1 & $1.000 \times 10^{0}$ & 7.00 & 5.50 & 7.00 & 5.50 \\
0.1 & $1.000 \times 10^{1}$ & 16.75 & 10.50 & 16.75 & 10.50 \\
0.01 & $1.000 \times 10^{2}$ & 52.75 & 29.00 & 52.50 & 29.00 \\
0.001 & $1.000 \times 10^{3}$ & 168.50 & 91.25 & 153.50 & 82.75 \\
\bottomrule
\end{tabular}
\caption{Three-dimensional manufactured solution: average number of iterations per time step for varying $\gamma_1$ with fixed $\gamma_2=1$.}
\label{tab:manufactured_gamma_scan_3d}
\end{table}

Varying $\gamma_2$ at $\gamma_1=1$ (Table~\ref{tab:manufactured_scaling2_scan_3d}) drives the solid towards incompressibility (Poisson ratio approaching $1/2$) and lowers the permeability, resulting in weak coupling and consequently small iteration counts, as observed in the two-dimensional case.

\begin{table}[ht!]
\renewcommand{\arraystretch}{1.15}
\centering
\begin{tabular}{ccccccc}
\toprule
$\gamma_2$ & Poisson ratio & coupling strength $\tau$ & \multicolumn{2}{c}{$k=1$} & \multicolumn{2}{c}{$k=2$} \\
 &  &  & $\beta_\mathrm{FS}$ & $\beta_\mathrm{max}$ & $\beta_\mathrm{FS}$ & $\beta_\mathrm{max}$ \\
\midrule
10000 & 0.49995 & $1.667 \times 10^{-4}$ & 3.00 & 4.00 & 3.00 & 4.00 \\
1000 & 0.4995 & $1.666 \times 10^{-3}$ & 4.00 & 4.75 & 4.00 & 4.75 \\
100 & 0.49505 & $1.656 \times 10^{-2}$ & 4.25 & 5.25 & 4.25 & 5.25 \\
10 & 0.45455 & $1.562 \times 10^{-1}$ & 6.00 & 6.00 & 6.00 & 6.00 \\
1 & 0.25 & $1.000 \times 10^{0}$ & 7.00 & 5.50 & 7.00 & 5.50 \\
0.1 & 0.04545 & $2.174 \times 10^{0}$ & 5.00 & 4.25 & 5.00 & 4.25 \\
0.01 & 0.00495 & $2.463 \times 10^{0}$ & 4.00 & 3.50 & 4.00 & 3.50 \\
\bottomrule
\end{tabular}
\caption{Three-dimensional manufactured solution: average number of iterations per time step for varying $\gamma_2$ with fixed $\gamma_1=1$.}
\label{tab:manufactured_scaling2_scan_3d}
\end{table}

In both scans the counts are set by the coupling strength and the tuning, essentially independently of the polynomial degree, exactly as in two dimensions. The maximal destabilization $\beta_\mathrm{max}$ helps most under strong coupling and is neutral to marginally adverse under weak coupling, so that a higher discretization order results in higher spatial accuracy without affecting the relative cost of the iterative coupling.

\answer{\subsection{A three-dimensional footing problem}\label{subsec:footing_3d}}
To test the splitting on a genuinely three-dimensional problem with mixed boundary conditions, and thereby to probe the boundary-condition independence of the analysis (Lemma~\ref{lemma:splitting}, Remark~\ref{remark:bc}), we consider a 3D footing problem, cf.\ cases I-III in~\cite{both2019gradient}. On the unit cube the solid is clamped at the base, loaded by a ramped normal traction on a central patch of the top face, drained on the remaining boundary, and sealed on the loaded parts, with no body force:
\[
\displacement=\mathbf{0} \ \text{on}\ [0,1]^2\times\{0\},
\qquad
\totalstress\mathbf n = 10^{9}\,t\,\mathbf e_3 \ \text{on}\ [0.25,0.75]^2\times\{1\},
\qquad
\totalstress\mathbf n = 0 \ \text{on the rest of}\ \partial\Omega,
\]
\[
\velocity\cdot\mathbf n = 0 \ \text{on}\ [0,1]^2\times\{0\} \cup [0.25,0.75]^2\times\{1\},
\qquad
p = 0 \ \text{on the rest of}\ \partial\Omega.
\]
In the dual-mixed finite element setting the boundary traction $\totalstress\mathbf n$ is essential on the stress space. It is homogeneous on the sides and the unloaded top, and non-homogeneous on the patch, where it enters as a time-dependent stress lifting on the mechanics right-hand side. The no-flow condition is essential on the flux space, whereas the clamping $\displacement=\mathbf{0}$ and the drainage $p=0$ are natural and, being homogeneous, drop out. The essential constraints are imposed by restricting the stress and flux blocks to their free degrees of freedom.

We adopt the material properties of~\cite{both2019gradient} ($\nu=0.2$, $\alpha=1$, $c_0=10^{-11}$, and $\kappa=10^{-13}$) and sweep the coupling through the Young modulus $E\in\{10^{9},10^{10},10^{11},10^{12}\}$ (smaller $E$, stronger coupling), with resulting coupling strength $\tau\in [0.18, 180]$. These material parameters generate assembled stiffness matrices whose entries span roughly 22 orders of magnitude. To circumvent scaling issues, all computations are therefore carried out in non-dimensional form by a change of units to moderate size.

We take five implicit-Euler steps of $\Delta t=0.1$ on a fixed structured tetrahedral mesh that resolves the patch edges, and report $k=1,2$. As for the Mandel problem, we test different values for the destabilization, again sweeping the fraction $\theta:=\beta/\beta_\mathrm{max}\in\{0,\tfrac14,\tfrac12,\tfrac34,1\}$, so that $\theta=0$ is the plain split and $\theta=1$ the theoretical optimum $\beta_\mathrm{max}$, which Lemma~\ref{lemma:splitting} identifies as the maximal destabilization with guaranteed robustness. The average iteration counts are reported in Tables~\ref{tab:footing_dim_beta_grid_3d_k1}--\ref{tab:footing_dim_beta_grid_3d_k2} for k=1,2.

\begin{table}[ht!]
\renewcommand{\arraystretch}{1.15}
\centering
\begin{tabular}{ccccccc}
\toprule
$E$ & coupling strength $\tau$ & \multicolumn{5}{c}{$\theta=\beta/\beta_\mathrm{max}$} \\
\cmidrule(lr){3-7}
 &  & $0$ & $0.25$ & $0.5$ & $0.75$ & $1$ \\
\midrule
$1 \times 10^{12}$ & $1.800 \times 10^{-1}$ & 7.00 & 6.40 & \textbf{6.20} & 7.00 & 7.20 \\
$1 \times 10^{11}$ & $1.800 \times 10^{0}$ & 13.20 & 11.80 & \textbf{10.40} & 12.40 & 17.60 \\
$1 \times 10^{10}$ & $1.800 \times 10^{1}$ & 20.40 & 17.80 & \textbf{15.20} & 22.60 & 92.20 \\
$1 \times 10^{9}$ & $1.800 \times 10^{2}$ & 22.40 & 19.40 & \textbf{16.40} & 26.40 & 647.20 \\
\bottomrule
\end{tabular}
\caption{Three-dimensional footing benchmark with the dimensional parameters of~\cite{both2019gradient}: average number of fixed-stress iterations per time step versus the destabilization fraction $\theta=\beta/\beta_\mathrm{max}$ for $k=1$. Here $\theta=0$ is the plain split and $\theta=1$ the theoretical optimum $\beta_\mathrm{max}$; the practical minimum is in bold.}
\label{tab:footing_dim_beta_grid_3d_k1}
\end{table}

\begin{table}[ht!]
\renewcommand{\arraystretch}{1.15}
\centering
\begin{tabular}{ccccccc}
\toprule
$E$ & coupling strength $\tau$ & \multicolumn{5}{c}{$\theta=\beta/\beta_\mathrm{max}$} \\
\cmidrule(lr){3-7}
 &  & $0$ & $0.25$ & $0.5$ & $0.75$ & $1$ \\
\midrule
$1 \times 10^{12}$ & $1.800 \times 10^{-1}$ & 7.00 & 6.40 & \textbf{6.20} & 7.00 & 7.20 \\
$1 \times 10^{11}$ & $1.800 \times 10^{0}$ & 13.20 & 11.80 & \textbf{10.20} & 12.40 & 17.40 \\
$1 \times 10^{10}$ & $1.800 \times 10^{1}$ & 20.40 & 17.80 & \textbf{15.00} & 22.20 & 84.80 \\
$1 \times 10^{9}$ & $1.800 \times 10^{2}$ & 22.40 & 19.40 & \textbf{16.40} & 26.20 & 542.00 \\
\bottomrule
\end{tabular}
\caption{Three-dimensional footing benchmark with the dimensional parameters of~\cite{both2019gradient}: average number of fixed-stress iterations per time step versus the destabilization fraction $\theta=\beta/\beta_\mathrm{max}$ for $k=2$. Here $\theta=0$ is the plain split and $\theta=1$ the theoretical optimum $\beta_\mathrm{max}$; the practical minimum is in bold.}
\label{tab:footing_dim_beta_grid_3d_k2}
\end{table}

Four observations stand out. First, the plain split $\theta=0$ is far more robust here than for the problem with the manufactured solution, its count rising only from $7.00$ to $22.40$ across three orders of coupling strength. On this configuration, destabilization is still beneficial and can mildly reduce iteration counts. Second, the counts are unimodal in $\theta$ with a minimum at $\theta=\tfrac12$, and the location of that minimum does not move with the coupling strength or with the polynomial degree; the gain over the plain split grows with the coupling, from about $11\%$ at $\tau\approx0.18$ to about $27\%$ at $\tau\approx180$. Third, the ascent beyond the minimum is steep and one-sided: at $\theta=\tfrac34$ the scheme is already slower than the plain split once $\tau\gtrsim18$, and at $\theta=1$ it degrades from neutral at weak coupling ($7.20$ versus $7.00$) to a significant increase of iterations at the strongest coupling ($647.20$ versus $22.40$ for $k=1$). This is a concrete instance of Remark~\ref{remark:optimality}: $\beta_\mathrm{max}$ is the a priori-robust maximal destabilization, and on this boundary-condition and stress configuration it sits at the very edge of the useful range rather than at its center. Fourth, the counts are essentially order-independent for $\theta\le\tfrac34$, agreeing to within $0.4$ iterations between $k=1$ and $k=2$; the only visible order effect is a partial relief of the over-destabilized column $\theta=1$ at strong coupling ($647.20\to542.00$ and $92.20\to84.80$), consistent with the discrete stability constants entering the practical optimum but not the theoretical contraction factor.

Finally, this example shows that the splitting converges reliably under genuine mixed boundary conditions (Remark~\ref{remark:bc}), and that the practically optimal destabilization parameter lies in between the plain split and the maximal value derived through the theory (Remark~\ref{remark:optimality}).

\section{Concluding remarks\label{sec:conclusion}}

In this work, we considered the fully-mixed formulation of the Biot equations, describing two-way coupled flow and deformation in porous media. This formulation strictly enforces the conservation of mass and linear momentum. Moreover, the cross-physics coupling is symmetric which we highlight two-fold. We highlight the flexibility in designing families of mixed finite elements, solely requiring the stability within the single subphysics, and then resulting in overall inf-sup stability for the coupled problem. Moreover, the symmetric coupling allows for effective iterative solution by employing a naive decoupling of flow and deformation, which can be identified with the common fixed-stress split. However, in addition we theoretically show that destabilization, i.e., the application of negative (contrary to positive) stabilization, results in improved convergence, which is consistent with previous studies of the fixed-stress split for the two-field formulation reporting that the full stabilization associated with "fixing the stress" can be reduced without sacrificing performance~\cite{mikelic2013convergence,BothBorregalesNordbottenKumarRadu:2017}. In line with previous observations, picking the optimal tuning depends on various problem characteristics~\cite{both2017numerical}, suggesting optimization either informed by theory~\cite{storvik2019optimization} or data-driven approaches~\cite{zabegaev2024automated}. Our analysis structurally differs from previous analyses of splitting schemes for symmetrically coupled problems~\cite{nuca2024splitting,brun2020iterative}, opening new views on the numerical analysis of such.

\section*{Acknowledgments}
JWB acknowledges support from the FRIPRO project “Unlocking maximal geological CO2 storage through experimentally validated mathematical modeling of dissolution and convective mixing (TIME4CO2)”, grant nr. 355188, funded by the Research Council of Norway.

\section*{Code Availability Statement}
The code used to generate the findings of this study is openly available in GitHub at \url{https://github.com/tugaydagli/splitting_4_arbitrary_mfem_biot} and archived in Zenodo at \url{https://doi.org/10.5281/zenodo.23800982}.

\bibliography{ref} 
\end{document}

%% file: defs.tex
\renewcommand{\div}{\operatorname{div}}
\newcommand{\tr}{\operatorname{tr}}
\newcommand{\curl}{\operatorname{curl}}

\newcommand{\storage}{{c_0}}
\newcommand{\bI}{\underline{\bf I}}
\newcommand{\tri}{{\mathcal T}}

\newcommand{\velocity}{{\bf { w}}}
\newcommand{\integratedvelocity}{{\bf {\tilde w}}}
\newcommand{\stress}{\underline{\boldsymbol{\sigma}}_\mathrm{eff}}
\newcommand{\totalstress}
{\underline{\boldsymbol{\sigma}}}

\newcommand{\displacement}{{\bf{u}}}
\newcommand{\vorticity}{\boldsymbol{\xi}}
\newcommand{\strain}{\underline{\boldsymbol \varepsilon}(\displacement)}
\newcommand{\testvorticity}{{\boldsymbol \eta}}
\newcommand{\testvelocity}{{\bf z}}
\newcommand{\teststress}{\underline{\boldsymbol{\tau}}}
\newcommand{\testdisplacement}{{\bf{v}}}
\newcommand{\testpressure}{q}

\newcommand{\integratedvelocityT}{{\bf {\tilde w}}_\tri}
\newcommand{\totalstressT}{\underline{\boldsymbol{\sigma}}_\tri}
\newcommand{\displacementT}{{\bf{u}}_\tri}
\newcommand{\vorticityT}{\boldsymbol{\xi}_\tri}
\newcommand{\pressureT}{p_\tri}

\newcommand{\testvorticityT}{{\boldsymbol \eta}_\tri}
\newcommand{\testvelocityT}{{\bf z}_\tri}
\newcommand{\teststressT}{\underline{\boldsymbol{\tau}}_\tri}
\newcommand{\testdisplacementT}{{\bf{v}}_\tri}
\newcommand{\testpressureT}{q_\tri}

\newcommand{\Hdiv}{\mathbb{H}(\operatorname{div}, \Omega)}
\newcommand{\hdiv}{{\bf {H}}(\operatorname{div}, \Omega)}

\newcommand{\Stress}{{\underline{\boldsymbol{\Sigma}}}}
\newcommand{\Pressure}{\mathrm{P}}
\newcommand{\Velocity}{{\bf{W}}}
\newcommand{\Displacement}{\bf{U}}
\newcommand{\Vorticity}{\boldsymbol\Xi}

\newcommand{\StressT }{{\underline{\boldsymbol{\Sigma}}_\tri}}
\newcommand{\VelocityT }{{\bf{W}}_\tri}
\newcommand{\DisplacementT }
{{\bf{U}}_\tri}
\newcommand{\PressureT }{\mathrm{P}_\tri}
\newcommand{\VorticityT }{\boldsymbol\Xi_\tri}

\newcommand{\RT}
{\boldsymbol{R} \boldsymbol{T}}
\newcommand{\ABF}
{\boldsymbol{A} \boldsymbol{B} \boldsymbol{F}}
\newcommand{\BDM}{\boldsymbol{B} \boldsymbol{D}
\boldsymbol{M}}
\newcommand{\BDFM}{\boldsymbol{B} \boldsymbol{D}\boldsymbol{F}
\boldsymbol{M}}
\newcommand{\BDDF}{\boldsymbol{B} \boldsymbol{D}\boldsymbol{D}
\boldsymbol{F}}

\newtheorem{remark}{Remark}

\newcommand{\tuplep}{\boldsymbol{\mathfrak{p}}}
\newcommand{\tupleq}{\boldsymbol{\mathfrak{q}}}
\newcommand{\tupleu}{\boldsymbol{\mathfrak{u}}}
\newcommand{\tuplev}{\boldsymbol{\mathfrak{v}}}

\newcommand{\tupleuT}{\tupleu_\tri}
\newcommand{\tuplevT}{\tuplev_\tri}
\newcommand{\tuplepT}{\tuplep_\tri}
\newcommand{\tupleqT}{\tupleq_\tri}

\newcommand{\tupleQ}{\boldsymbol{\mathfrak{Q}}}
\newcommand{\tupleV}{\boldsymbol{\mathfrak{U}}}
\newcommand{\tupleVT}{\tupleV_\tri}
\newcommand{\tupleQT}{\tupleQ_\tri}

\newcommand{\PsiT}{\boldsymbol{\Psi}_\tri}